\documentclass[]{interact}

\usepackage{epstopdf}% To incorporate .eps illustrations using PDFLaTeX, etc.
\usepackage[caption=false]{subfig}% Support for small, `sub' figures and tables
\theoremstyle{plain}% Theorem-like structures provided by amsthm.sty
\newtheorem{theorem}{Theorem}[section]
\newtheorem{lemma}[theorem]{Lemma}

\theoremstyle{definition}

\theoremstyle{remark}
\newtheorem{remark}{Remark}

\usepackage{enumitem}
\newcommand{\R}{ \mathbb{R} }

\newcommand{\C}{ \mathbb{C} }
\newcommand{\Y}{ \widehat{Y} }
\newcommand{\gam}{{\pmb{\gamma}}}
\newcommand{\bg}{{\pmb{g}}}
\newcommand{\Span}{ {\rm span} }
\newcommand{\Rea}{ {\rm Re} }
\newcommand{\Ima}{ {\rm Im} }
\newcommand{\esssup}{ {\rm ess \: sup} }
\begin{document}

\articletype{}% Specify the article type or omit as appropriate

\title{A theoretical investigation of time-dependent Kohn--Sham
  equations: new proofs}

\author{
\name{G. Ciaramella\textsuperscript{a}\thanks{ 
CONTACT A.~N. Author. Email: gabriele.ciaramella@uni-konstanz.de} 
and M. Sprengel\textsuperscript{b} and A. Borzi\textsuperscript{c}}
\affil{\textsuperscript{a}Fachbereich Mathematik, Universit\"at Konstanz,
\textsuperscript{b}---
\textsuperscript{c}Institut f\"ur Mathematik, Universit\"at W\"urzburg,
}
}

\maketitle

\begin{abstract}
In this paper, a new analysis for existence, uniqueness, and regularity of solutions to
a time-dependent Kohn-Sham equation is presented. 
The Kohn-Sham equation is a nonlinear integral Schr\"odinger equation
that is of great importance in many applications in physics and computational chemistry.
To deal with the time-dependent, nonlinear and non-local potentials of the Kohn-Sham equation,
the analysis presented in this manuscript makes use of
energy estimates, fixed-point arguments, regularization techniques, and
direct estimates of the non-local potential terms.
The assumptions considered for the time-dependent and nonlinear potentials 
make the obtained theoretical results suitable to be used also in an optimal control framework.
\end{abstract}

% REQUIRED
\begin{keywords}
  Kohn-Sham model; time-dependent density functional theory;
  nonlinear \linebreak Schr\"odinger equation
\end{keywords}

\section{Introduction}
One of the main issues in computational chemistry and physics is the curse of dimensionality
of the multi-particle Schr\"odinger equation. To tackle this problem the so-called
density-functional theory has been introduced by P. Hohenberg and W. Kohn in 1964,
and W. Kohn and L. J. Sham in 1965; see \cite{HohenbergKohn:PhysRev:1964,KohnSham1965}.
This theory has been extended to time-dependent problems by E. Runge and E. K. U. Gross
in 1984 \cite{RungeGross1984}; see also \cite{vanLeeuwen99,LectureNotesTDDFT,Rugg2015}.
These theories allow one to describe the state of a multi-particle physical system,
represented by the solution of the multi-particle Schr\"odinger equation, by a density
function corresponding to a system of nonlinear integral one-particle Schr\"odinger equations.
This is the time-dependent Kohn-Sham (TDKS) system of equations that allows one to describe the state a system of
$N$ particles. In particular, the Kohn-Sham system is a set of $N$ one-particle Schr\"odinger equations,
which is numerically tractable, in contrast to the full multi-particle Schr\"odinger equation;
see, e.g., \cite{vanLeeuwen99,LectureNotesTDDFT,LIBRO_QUANTUM}. For this reason, the TDKS model became
central in many applications in computational chemistry and physics dealing also with optimal control problems,
see, e.g., \cite{vanLeeuwen99,LectureNotesTDDFT,octopuscite,CastroWerschnikGross2012,Sprengel2,Sprengel3,LIBRO_QUANTUM}
and references therein.

Nonlinear integral Schr\"odinger equations motivated also great interests in the mathematical
community; see, e.g., \cite{CancesLeBris1999} and the classical reference \cite{cazenave2003semilinear}.
In these works, nonlinear integral Schr\"odinger equations with time-independent potentials are treated
using classical semi-group theory. We remark that the results of \cite{CancesLeBris1999,cazenave2003semilinear} cannot be
applied to our problem as we consider explicitly time-dependent potentials that are not covered in these references.

The classical approach based on semi-group theory, has been extended to the case of time-dependent
potential. We refer to \cite[Chapters 6 and 7]{1983approximation} and, e.g., to the very recent
work \cite{MASPERO2017721}. However, to use these results the potentials must be continuous functions
in time; see, e.g., Hypothesis (3) of Theorem 6.2.5 in \cite{1983approximation} and Hypothesis H0 in \cite{MASPERO2017721}. 
This assumption is in general not suitable for concrete applications
and in optimal control theory, where the time-dependency of the potentials is due to time-dependent
control functions that are, in general, much less regular than continuously differentiable functions.

To the best of our knowledge, a time-dependent model similar to the one considered in this work,
is only addressed in \cite{Jerome2015,Jerome2019,Jerome2013,Sprengel1}.
In \cite{Jerome2015,Jerome2019,Jerome2013} the authors prove
existence and uniqueness of solutions assuming that the potentials are continuously
differentiable in time. As mentioned above, this assumption appears too strong in the
context of control applications.
An attempt to improve these results is made in \cite{Sprengel1}, where the authors try to obtain
existence and regularity results by exploiting a Galerkin approach combined with energy estimates.
Unfortunately, the proofs of some energy estimates derived in \cite{Sprengel1} are erroneous.
The goal of this paper is to remedy to this issue. 
In this work, we prove results that are very similar to
the ones claimed in \cite{Sprengel1}, by slightly strengthening the assumptions
on the boundary regularity of the space domain and on the regularity of the nonlinear potentials.
Our results are obtained using proof techniques that are different from the ones used in \cite{Sprengel1}.

To be more specific, the goal of this work is to prove existence, uniqueness and regularity of solutions to the TDKS equation
\begin{equation}\label{eq:TDKS_pre}
\begin{split}
i \partial_t \Psi(x,t) &= - \Delta \Psi(x,t) + V(x,t)\Psi(x,t) + F(\Psi(x,t)) \text{ in $\Omega \times (0,T)$},\\
\Psi(x,0) &= \Psi_0(x) \text{ in $\Omega$}, \\
\Psi(x,t) &= 0 \text{ on $\partial \Omega \times (0,T)$},
\end{split}
\end{equation}
where $\Omega \subset \R^3$ is a bounded domain and $\Psi(x,t) \in \C$. More details about this model
are discussed in Section \ref{sec:formulation}, where we also state our main assumptions.
Notice that \eqref{eq:TDKS_pre} is a single Schr\"odinger equation. The choice of having
this single equation, rather than a system of Schr\"odinger equations is only made to conveniently
ease the notation. However, the extension of our results to the case of a system of Schr\"odinger
equations is straightforward.

Our theoretical analysis proceeds as follows. In Section \ref{sec:auxiliary}, we study an auxiliary problem,
namely a linear inhomogeneous Schr\"odinger equation with time-dependent potentials.
We prove existence and uniqueness of solutions to this equation by a Galerkin approach and energy estimates.
These existence results and the corresponding energy estimates are then used in Section \ref{sec:full_NL}
for proving existence and uniqueness of a solution to \eqref{eq:TDKS_pre} through a fixed-point argument.
The solution obtained in Section \ref{sec:full_NL} is `regular', in the sense that it lies in the space
$$L^\infty(0,T;H^1_0(\Omega;\C) \cap H^2(\Omega;\C)) \: \cap \: C([0,T];H^1_0(\Omega;\C)),$$
while its time weak derivative is in $L^{\infty}(0,T;L^2(\Omega;\C))$.
This result is obtained by requiring in \eqref{eq:TDKS_pre} that the boundary $\partial \Omega$ is of class $C^{2,1}$,
that the potentials $V$ and $F$ are twice differentiable in space, and that the initial condition function $\Psi_0$
is in $H^1_0(\Omega;\C)\cap H^2(\Omega;\C)$.
These hypotheses are relaxed in Section \ref{sec:less_regular}, where
we assume that the potentials are differentiable and the initial condition is in $H^1_0(\Omega)$
and prove existence and uniqueness of
solution in $L^2(0,T;H^1_0(\Omega;\C))$ with time weak derivative in $L^2(0,T;H^{-1}(\Omega;\C))$.
This is achieved by combining the results of Sections \ref{sec:auxiliary}
with a regularization technique based on mollifiers.

Finally, we wish to remark that the results presented in this work can be extended to the case of the 
Kohn-Sham adjoint equation, which is used in the framework of optimal control problems governed by
the TDKS equation. This adjoint equation can be regarded, in some sense, as a linearized
TDKS equation with inhomogeneous right-hand side; see, e.g., \cite{Sprengel3,Sprengel2,Sprengel1}.

\section{Formulation of the TDKS problem and main assumptions}\label{sec:formulation}
The goal of this work is to prove existence and uniqueness of solutions
to the TDKS equation
\begin{equation}\label{eq:TDKS}
\begin{split}
i \partial_t \Psi(x,t) &= - \Delta \Psi(x,t) + V(x,t)\Psi(x,t) + F(\Psi(x,t)) \text{ in $\Omega \times (0,T)$},\\
\Psi(x,0) &= \Psi_0(x) \text{ in $\Omega$}, \\
\Psi(x,t) &= 0 \text{ on $\partial \Omega \times (0,T)$},
\end{split}
\end{equation}
where $\Omega \subset \R^3$ is a bounded domain and
$\Psi(x,t) \in \C$. The external potential $V$ is
\begin{equation}
V(x,t) = V_0(x) + V_u(x) u(t) \in \R,
\end{equation}
where $V_0$, $V_u$ and $u$ are real functions. The nonlinearities of the model
are expressed through $F$, that is
\begin{equation}\label{eq:F_NL}
F(\Psi) = V_H(\Psi) \Psi + V_{xc}(\Psi) \Psi,
\end{equation}
where $V_H$ is the Hartree potential
\begin{equation}
V_H(\Psi)(x,t) = \int_{\R^3} \frac{|\mathcal{E}\Psi(y,t)|^2}{|x-y|}dy,
\end{equation}
and $V_{xc}(\Psi(x,t)) \in \R$ takes into account exchange and correlation potentials, whose dependence on
time and space variables is implicit through $\Psi$.
For further details about the Kohn-Sham model, we refer to \cite{Sprengel1} and references therein.
In the Hartree potential, we introduced the operator $\mathcal{E}$ that extends $\Psi$
from $\Omega$ to $\R^3$.
Since $\partial \Omega$ is assumed to be of class $C^{2,1}$, this operator can be defined
using the so-called Calder\'on extension theory, as done in Section IV of \cite{adams1970sobolev}
(see in particular Theorem 4.32).

Let $(\cdot , \cdot)$ denotes the usual inner product for $L^2(\Omega;\C)$ and $\| \cdot \|$ the corresponding induced norm.
The weak form of the TDKS equation is
\begin{equation}\label{eq:TDKS_weak}
i (\partial_t \Psi,\Phi) = (\nabla \Psi, \nabla \Phi) + (V \Psi,\Phi) + (F(\Psi),\Phi) \text{ a.e. in $(0,T)$},
\end{equation}
for all $\Phi \in H^1_0(\Omega;\C)$. 
We look for a weak solution $\Psi$ to \eqref{eq:TDKS_weak} that satisfies the initial
condition $\Psi(\cdot,0) = \Psi_0$.

Consider the following Banach spaces and the corresponding norms
\begin{align*}
X &:= L^2(0,T;H^1_0(\Omega;\C)), & \| \Phi \|_X^2 &:= \int_0^T \| \Phi(t) \|_{H^1}^2 dt, \\
X^* &:= L^2(0,T;H^{-1}(\Omega;\C)), & \| \Phi \|_{X^*}^2 &:= \int_0^T \| \Phi(t) \|_{H^{-1}}^2 dt,\\
W(0,T) &:= \{ \Phi \in X \, : \, \Phi' \in X^* \}, & \| \Phi \|_{W}^2 &:= \| \Phi \|_X^2 + \| \Phi' \|_{X^*}^2,\\
Y &:=L^2(0,T;L^2(\Omega;\C)), & \| \Phi \|_Y^2 &:= \int_0^T \| \Phi(t) \|^2 dt,\\
Z &:= H^1_0(\Omega;\C) \cap H^2(\Omega;\C), & \| \Phi \|_{H^2}^2 &:= \sum_{0\leq|\pmb{\alpha}|\leq 2} \|\partial^{\pmb{\alpha}}\Phi\|^2,\\
Y_{\infty,0} &:=L^\infty(0,T;L^2(\Omega;\C)), & \| \Phi \|_{Y_{\infty,0}} &:= \esssup_{t \in (0,T)} \| \Phi(t) \|, \\
Y_{\infty,1} &:=L^\infty(0,T;H^1(\Omega;\C)), & \| \Phi \|_{Y_{\infty,1}} &:= \esssup_{t \in (0,T)} \| \Phi(t) \|_{H^1}, \\
\Y &:=L^\infty(0,T;Z), & \| \Phi \|_{\Y} &:= \esssup_{t \in (0,T)} \| \Phi(t) \|_{H^2}.
\end{align*}
Let us introduce our standing assumptions:
\begin{itemize}
\item[{\rm (A1)}]\label{ass:A1} The domain $\Omega \subset \R^3$ is bounded and $\partial \Omega \in C^{2,1}$.
\item[{\rm (A2)}]\label{ass:A2} $V_0,V_u \in W^{2,\infty}(\Omega;\R)$, where $W^{2,\infty}(\Omega;\R)$ is a standard Sobolev space;
see, e.g., \cite{Ciarlet2013}.
\item[{\rm (A3)}]\label{ass:A3} $u \in L^\infty(0,T;\R)$.
\item[{\rm (A4)}]\label{ass:A4} $\Psi_0 \in Z$.
\item[{\rm (A5)}]\label{ass:A5} For every $\Phi \in Z$ it holds that $V_{xc}(\Phi)\Phi \in Z$ and there exist positive constants $K$ and $\widetilde{K}$ such that
\begin{equation}\label{eq:Lipschitz2}
\| V_{xc}(\Phi)\Phi - V_{xc}(\Lambda)\Lambda \| \leq K \| \Phi - \Lambda \|
\end{equation}
\begin{equation}\label{eq:Lipschitz3}
\| V_{xc}(\Phi)\Phi - V_{xc}(\Lambda)\Lambda \|_{H^2} \leq \widetilde{K} \| \Phi - \Lambda \|_{H^2}
\end{equation}
for any $\Phi,\Lambda \in Z$.
%constants $K_1,K_2,K_3>0$ such that
%\begin{equation*}
%\begin{split}
%\| V_{xc}(\Phi)\Phi - V_{xc}(\Lambda)\Lambda \| &\leq K_1 \| \Phi - \Lambda \|, \\
%\| \nabla (V_{xc}(\Phi)\Phi) - \nabla (V_{xc}(\Lambda)\Lambda) \| &\leq K_2 \| \nabla \Phi - \nabla \Lambda \|, \\
%\| \Delta (V_{xc}(\Phi)\Phi) - \Delta (V_{xc}(\Lambda)\Lambda) \| &\leq K_3 \| \Delta\Phi - \Delta\Lambda \|, \\
%\end{split}
%\end{equation*}
%for any $\Phi,\Lambda \in Z$.
\end{itemize}
Notice that the assumptions {\rm (A2)}, {\rm (A4)} and {\rm (A5)} will be relaxed in Section \ref{sec:less_regular}.

Let us recall some facts and existing results that we will use in this work.
\begin{itemize}
\item[{\rm (B1)}]\label{ass:B1} There exists an orthogonal basis for $H^1_0(\Omega;\C)$ which is orthonormal in \linebreak $L^2(\Omega;\C)$.
Since $\partial \Omega \in C^{2,1}$ we can choose this basis to be $\{ \Psi_j \}_j$, 
where $\Psi_j \in Z \cap H^3(\Omega;\C)$ are eigenfunctions of the Laplace
operator. This follows from \cite[Theorem 1 in 6.5.1 and Theorem 4 in 6.3.2]{Evans2010}
and \cite[Theorem 2.5.1.1]{grisvard2011elliptic}.
Throughout this paper $\{ \Psi_j \}_j$ is used to denote this basis.

\item[{\rm (B2)}]\label{ass:B2} For any integer $m>0$ and some coefficients $\gamma_1,\dots,\gamma_m \in \C$ the functions
$\widehat{\Psi}:= \sum_{j=1}^m \gamma_j \Psi_j$ and $\Delta \widehat{\Psi}$ vanish on $\partial \Omega$.

\item[{\rm (B3)}]\label{ass:B3} For any $\Psi_0 \in Z$, we can define $(\gamma_0)_j := (\Psi_0,\Psi_j)$, for $j=1,\dots,M$ and $\widehat{\Psi}_0 :=  \sum_{j=1}^m (\gamma_0)_j \Psi_j$.
Then the inequalities $\| \widehat{\Psi}_0 \|^2 \leq \| \Psi_0 \|^2$, $\| \Delta \widehat{\Psi}_0 \|^2 \leq \| \Delta \Psi_0 \|^2$
and $\| \nabla \widehat{\Psi}_0 \|^2 \leq \| \nabla \Psi_0 \|^2$ 
follow by Parseval-Plancherel's theorem and the orthogonality properties of $\Psi_j$.
\item[{\rm (B4)}]\label{ass:B4} Consider the extension operator $\mathcal{E}$.
Since $\partial \Omega \in C^{2,1}$, Theorem 4.32 in Section IV of \cite{adams1970sobolev}
guarantees that $\mathcal{E}$ is a continuous operator from  $W^{1,p}(\Omega;\C)$ to $W^{1,p}(\R^3;\C)$
and from $W^{2,p}(\Omega;\C)$ to $W^{2,p}(\R^3;\C)$ for $1< p < \infty$. 
\item[{\rm (B5)}]\label{ass:B5} Consider the Hartree potential $V_H$ and define $f(\Phi):=V_H(\Phi)\Phi$. 
It follows from \cite[Lemma 5]{CancesLeBris1999} that
there exist positive constants $C_a,C_b,C_c$ such that
\begin{equation}\label{eq:Lipschitz1}
\| f(\Phi) - f(\Lambda) \| \leq C_a (\| \Phi \|_{H^1}^2 + \| \Lambda \|_{H^1}^2) \| \Phi - \Lambda \| \quad \forall \Phi,\Lambda \in H^1(\Omega;\C),
\end{equation}
\begin{equation}
\| f(\Phi) \|_{H^2} \leq C_b \| \Phi \|_{H^1}^2 \| \Phi \|_{H^2} \quad \forall \Phi \in H^2(\Omega;\C),
\end{equation}
\begin{equation}\label{eq:11}
\| f(\Phi) - f(\Lambda) \|_{H^2} \leq C_c (\| \Phi \|_{H^2}^2 + \| \Lambda \|_{H^2}^2) \| \Phi - \Lambda \|_{H^2} \quad \forall \Phi,\Lambda \in H^2(\Omega;\C).
\end{equation}
\item[{\rm (B6)}]\label{ass:B6} Consider the space $Z$. The norms $\| \cdot \|_{H^2}$ and $\| \cdot \|_Z := \| \Delta \cdot \|$ are equivalent; see, e.g., \cite[Theorem 2.31]{libro}.
We denote by $C_Z$ the positive equivalence constant such that $\| \Phi \|_{H^2} \leq C_Z \| \Phi \|_Z$, $\forall \Phi \in Z$.
\end{itemize}

\section{An auxiliary problem}\label{sec:auxiliary}
Consider the auxiliary problem
\begin{equation}\label{eq:TDKS_aux}
\begin{split}
i \partial_t \Psi(x,t) &= - \Delta \Psi(x,t) + V(x,t)\Psi(x,t) + G(x,t) \text{ in $\Omega \times (0,T)$},\\
\Psi(x,0) &= \Psi_0(x) \text{ in $\Omega$}, \\
\Psi(x,t) &= 0 \text{ on $\partial \Omega \times (0,T)$},
\end{split}
\end{equation}
where $G \in \Y := L^\infty(0,T;Z)$ is a given function.
Problem \eqref{eq:TDKS_aux} in weak form is
\begin{equation}\label{eq:TDKS_weak_aux}
i (\partial_t \Psi,\Phi) = (\nabla \Psi, \nabla \Phi) + (V \Psi,\Phi) + (G,\Phi) \text{ a.e. in $(0,T)$ and $\forall \Phi \in H^1_0(\Omega;\C)$}.
\end{equation}
The goal of this section is to prove that this weak problem is uniquely solvable and to obtain energy estimates for the corresponding solution.
We proceed by using a Galerkin approach. To this purpose, let us consider an integer $m>0$, a finite-dimensional space $W_m:= \Span\{\Psi_1,\dots,\Psi_m\}$, where $\{\Psi_j\}_j$ is the basis introduced
in (B1), and the Galerkin approximation problem
\begin{equation}\label{eq:TDKS_weak_aux_m}
i (\partial_t \Psi^a_m,\Phi) = (\nabla \Psi^a_m, \nabla \Phi) + (V \Psi^a_m,\Phi) + (G,\Phi) \text{ a.e. in $(0,T)$ and $\forall \Phi \in W_m$}.
\end{equation}
If we make the ansatz
$\Psi^a_m(t) = \sum_{j=1}^m \gamma_j(t) \Psi_j$, where the coefficients $\gamma_j(t)$ are time-dependent functions,
then \eqref{eq:TDKS_weak_aux_m} is equivalent to the following initial-value problem
\begin{equation}\label{eq:ODE}
\begin{split}
i \dot{\gam}(t) &= A(t) \gam(t) + \bg(t) \text{ in $(0,T)$},\\
\gam(0) &= \gam_0, \\
\end{split}
\end{equation}
where $\gam(t) = [ \gamma_1(t) , \dots, \gamma_m(t) ]^\top$,
$\gam_0 \in \C^m$ is defined by $(\gam_0)_j := ( \Psi_0 , \Psi_j)$ for $j=1,\dots,m$,
and $A(t) \in \R^{m\times m}$ and $\bg(t) \in \R^m$ are obtained by projecting the right-hand side operators and functions of \eqref{eq:TDKS_aux} onto $W_m$.
The finite-dimensional problem \eqref{eq:ODE} is uniquely solvable by an absolutely continuous function $\gam$. This follows by Carath\'eodory's existence theorem
(since $A \in L^1(0,T;\R^{m\times m})$ and $\bg \in L^1(0,T;\R^m)$); see, e.g., \cite{Walter}.
Therefore, the Galerkin approximation $\Psi^a_m$ has the following regularity
\begin{equation*}
\Psi^a_m \in C([0,T];Z \cap H^3(\Omega)), \quad \partial_t \Psi^a_m \in L^1(0,T;Z \cap H^3(\Omega)).
\end{equation*}

We now prove the following energy estimates for the Galerkin solution $\Psi^a_m$.
\begin{theorem}(Energy estimates for the auxiliary problem)\label{thm:energy_aux}
Let $G \in \Y$. Then for almost all $t \in (0,T)$ there exist positive constants $C_{\nabla}$, $C_{1,\Delta}$, $C_{2,\Delta}$, $C_{3,\Delta}$, $K_{-1}$ and $K_{-2}$ (independent of $m$) such that
\begin{equation}\label{eq:EN_1}
\| \Psi^a_m(t) \|^2 \leq \exp(T) \Bigl[ \| \Psi_0 \|^2 + T \| G \|_{Y_{\infty,0}}^2 \Bigr],
\end{equation}
\begin{equation}\label{eq:EN_2}
\| \nabla \Psi^a_m(t) \|^2 \leq \exp(C_\nabla T) \Bigl[ \| \nabla \Psi_0 \|^2 + T \| G \|_{Y_{\infty,1}}^2 \Bigr],
\end{equation}
\begin{equation}\label{eq:EN_3}
\| \Delta \Psi^a_m(t) \|^2 \leq \exp(C_{1,\Delta} T) \Bigl[ \| \Delta \Psi_0 \|^2 
+ T \bigl[ C_{2,\Delta} \| \nabla \Psi_0 \|^2 + C_{3,\Delta} \| G \|_{\Y}^2 \bigr] \Bigr],
\end{equation}
\begin{equation}\label{eq:EN_4_pre}
\| \partial_t \Psi^a_m(t) \|_{H^{-1}} \leq K_{-1},%(\| G \|_{Y_{\infty,1}},\|\Psi_0\|_{H^1}),
\end{equation}
\begin{equation}\label{eq:EN_4}
\| \partial_t \Psi^a_m(t) \| \leq K_{-2},%(\| G \|_{\Y},\|\Psi_0\|_{H^2}),
\end{equation}
where $K_{-1}$ depends on  $\| G \|_{Y_{\infty,1}}$ and $\|\Psi_0\|_{H^1}$,
and $K_{-2}$ depends on $\| G \|_{\Y}$ and $\|\Psi_0\|_{H^2}$.
\end{theorem}

\begin{proof}
By testing \eqref{eq:TDKS_weak_aux_m} with $\Phi=\Psi^a_m$, we get
\begin{equation*}
i (\partial_t \Psi^a_m,\Psi^a_m) = (\nabla \Psi^a_m, \nabla \Psi^a_m) + (V \Psi^a_m,\Psi^a_m) + (G,\Psi^a_m),
\end{equation*}
whose imaginary part is
\begin{equation*}
\frac{1}{2} \frac{d}{dt} \| \Psi^a_m(t) \|^2 = \Ima( G(t) , \Psi^a_m(t)) 
\leq \| G(t) \| \| \Psi^a_m(t) \| 
\leq \frac{1}{2}\| G \|_{Y_{\infty,0}}^2 + \frac{1}{2}\| \Psi^a_m(t) \|^2.
\end{equation*}
The estimate \eqref{eq:EN_1} follows by Gr\"onwall's inequality in differential form and (B3).

To prove \eqref{eq:EN_2}, we compute
\begin{equation}\label{eq:proof_EN_2}
\frac{d}{dt} \| \nabla \Psi^a_m(t) \|^2 = 2 \Rea ( \nabla \partial_t \Psi^a_m(t) , \nabla \Psi^a_m(t) ) = 2 \Rea ( \partial_t \Psi^a_m(t) , - \Delta \Psi^a_m(t) ),
\end{equation}
where we integrated by parts and used the fact that $\partial_t \Psi^a_m(t)=0$ on $\partial \Omega$ for almost all $t$ in $(0,T)$.
Equation \eqref{eq:TDKS_weak_aux_m} in strong form is
\begin{equation}\label{eq:TDKS_str_aux_m}
\partial_t \Psi^a_m = i\Delta \Psi^a_m -i V \Psi^a_m -i G.
\end{equation}
Substituting \eqref{eq:TDKS_str_aux_m} into \eqref{eq:proof_EN_2}, we get
\begin{equation*}
\begin{split}
\frac{d}{dt} \| \nabla \Psi^a_m(t) \|^2 
&= 2 \Rea ( i\Delta \Psi^a_m(t) -i V(t) \Psi^a_m(t) -i G(t) , - \Delta \Psi^a_m(t) ) \\
&= 2 \Rea \Bigl[-i \| \Delta \Psi^a_m(t) \|^2 + i ( V(t) \Psi^a_m(t) , \Delta \Psi^a_m(t) ) + i ( G(t) , \Delta \Psi^a_m(t) )\Bigr] \\
&= 2 \Rea \Bigl[ i ( V(t) \Psi^a_m(t) , \Delta \Psi^a_m(t) ) + i ( G(t) , \Delta \Psi^a_m(t) )\Bigr] \\
&= 2 \Rea \Bigl[ -i ( \nabla(V(t) \Psi^a_m(t)) , \nabla \Psi^a_m(t) ) - i ( \nabla G(t) , \nabla \Psi^a_m(t) )\Bigr],
\end{split}
\end{equation*}
which implies that
\begin{equation*}
\begin{split}
\frac{d}{dt} \| \nabla \Psi^a_m(t) \|^2 
&\leq \| \nabla (V(t) \Psi^a_m(t)) \|^2 + \| \nabla G(t) \|^2 + 2 \| \nabla \Psi^a_m(t) \|^2 \\
&\leq 2\| \nabla V \|_{\infty}^2 \| \Psi^a_m(t) \|^2 + \| \nabla G(t) \|^2 + \bigl[ 2 + 2 \| V \|_{\infty}^2 \bigr]  \| \nabla \Psi^a_m(t) \|^2 \\
&\leq \| G \|_{Y_{\infty,1}}^2 + 2 \bigl[ 1 + \| V \|_{\infty}^2 + C_{PF}^2 \| \nabla V \|_{\infty}^2 \bigr]  \| \nabla \Psi^a_m(t) \|^2,
\end{split}
\end{equation*}
where $\| \cdot \|_{\infty}$ denotes the usual norm in $L^{\infty}(0,T;L^{\infty}(\Omega))$,
%we used the Poincar\'e-Friedrichs inequality, 
and $C_{PF}$ is the Poincar\'e-Friedrichs constant.
The estimate \eqref{eq:EN_2} follows by using Gr\"onwall's inequality in differential form, (B3) and setting
$C_{\nabla} =  2[1 + \| V \|_{\infty}^2 + C_{PF}^2 \| \nabla V \|_{\infty}^2]$.

To prove \eqref{eq:EN_3}, let us consider
\begin{equation}\label{eq:proof_EN_3}
\frac{d}{dt} \| \Delta \Psi^a_m(t) \|^2 = 2 \Rea ( \Delta \partial_t \Psi^a_m(t) , \Delta \Psi^a_m(t) ).
\end{equation}
Substituting \eqref{eq:TDKS_str_aux_m} into \eqref{eq:proof_EN_3}, we obtain
\begin{equation}\label{eq:DELTA}
\begin{split}
\frac{d}{dt} \| \Delta \Psi^a_m(t) \|^2 
&= 2 \Rea ( \Delta(\partial_t \Psi^a_m(t)) , \Delta \Psi^a_m(t) ) \\
&= 2 \Rea ( - \nabla (\partial_t \Psi^a_m(t)) , \nabla (\Delta \Psi^a_m(t) ) ) \\
&= 2 \Rea ( -\nabla (i\Delta \Psi^a_m(t) -i V \Psi^a_m(t) -i G(t)) , \nabla (\Delta \Psi^a_m(t) ) ) \\
%&= 2 \Rea \Bigl[ - i \| \nabla(\Delta \Psi^a_m(t)) \|^2  -i ( \Delta(V(t) \Psi^a_m(t)) , \Delta \Psi^a_m(t) ) -i ( \Delta 
%G(t) , \Delta \Psi^a_m(t) ) \Bigr]\\
&= 2 \Rea \Bigl[ - i \| \nabla(\Delta \Psi^a_m(t)) \|^2  -i ( \Delta(V(t) \Psi^a_m(t)) + \Delta G(t) , \Delta \Psi^a_m(t) ) \Bigr]\\
&= 2 \Rea \Bigl[ -i ( \Delta(V(t) \Psi^a_m(t)) , \Delta \Psi^a_m(t) ) -i ( \Delta G(t) , \Delta \Psi^a_m(t) ) \Bigr],\\
\end{split}
\end{equation}
where we used the fact that $\Psi^a_m(t) \in H^3(\Omega;\C)$ (see (B1)),
$\Psi^a_m(t)=0$, $\Delta \Psi^a_m(t)=0$ and $G(t)=0$ ($G(t) \in Z$)
on $\partial \Omega$ for almost all $t$ in $(0,T)$.
Hence, we get
\begin{equation*}
\begin{split}
\frac{d}{dt} \| \Delta \Psi^a_m(t) \|^2 
&= 2 \Rea \Bigl[ -i ( \Delta V(t) \Psi^a_m(t) + V(t) \Delta \Psi^a_m(t) + 2 \nabla V(t) \nabla \Psi^a_m(t), \Delta \Psi^a_m(t) ) \\
&\qquad \quad \; -i ( \Delta G(t) , \Delta \Psi^a_m(t) ) \Bigr]\\
&\leq 2 \| V \|_{\infty} \| \Delta \Psi^a_m(t) \|^2 + 2 \| \Delta \Psi^a_m(t) \| \| \Delta V(t) \Psi^a_m(t) + 2 \nabla V(t) \nabla \Psi^a_m(t) \| \\
&\quad +\| \Delta G(t) \|^2 + \| \Delta \Psi^a_m(t) \|^2 \\
&\leq \bigl[2+2\|V\|_{\infty} \bigr] \| \Delta \Psi^a_m(t) \|^2 + \| G \|_{\Y}^2 + 8 \| \nabla V \|_{\infty}^2 \| \nabla \Psi^a_m(t) \|^2 \\
&\quad + 2 \| \Delta V \|_{\infty}^2 \| \Psi^a_m(t) \|^2. \\
\end{split}
\end{equation*}
Using the Poincar\'e-Friedrichs inequality (to estimate $\| \Psi^a_m(t) \|^2$ by $\| \nabla \Psi^a_m(t) \|^2$) and \eqref{eq:EN_2}, we obtain
\begin{equation*}
\begin{split}
\frac{d}{dt} \| \Delta \Psi^a_m(t) \|^2 &\leq \bigl[2+2\|V\|_{\infty} \bigr] \| \Delta \Psi^a_m(t) \|^2 \\
&+ \bigl[ 8 \| \nabla V \|_{\infty}^2 + 2 C_{PF}^2 \| \Delta V \|_{\infty}^2 \bigr] \exp(C_\nabla T) \| \nabla \Psi_0 \|^2 \\
&+ \Bigl[ 1 + T \bigl[ 8 \| \nabla V \|_{\infty}^2 + 2 C_{PF}^2 \| \Delta V \|_{\infty}^2 \bigr] \exp(C_\nabla T) \Bigl] \| G \|_{\Y}^2 \\
&= C_{1,\Delta} \| \Delta \Psi^a_m(t) \|^2 + C_{2,\Delta} \| \nabla \Psi_0 \|^2 + C_{3,\Delta} \| G \|_{\Y}^2 , \\
\end{split}
\end{equation*}
where 
$$C_{1,\Delta} = 2+2\|V\|_{\infty},$$ 
$$C_{2,\Delta} = \bigl[ 8 \| \nabla V \|_{\infty}^2 + 2 C_{PF}^2 \| \Delta V \|_{\infty}^2 \bigr] \exp(C_\nabla T),$$
$$C_{3,\Delta} = \Bigl[ 1 + T \bigl[ 8 \| \nabla V \|_{\infty}^2 + 2 C_{PF}^2 \| \Delta V \|_{\infty}^2 \bigr] \exp(C_\nabla T) \Bigl].$$
The result \eqref{eq:EN_3} follows by using Gr\"onwall's inequality in differential form and (B3).

Next, we prove \eqref{eq:EN_4_pre}. Consider any $\Phi \in H^1_0(\Omega)$ such that $\| \Phi \|_{H^1}\leq 1$
and the equation \eqref{eq:TDKS_weak_aux_m}.
We then write
\begin{equation*}
\begin{split}
|\langle \partial_t \Psi^a_m(t),\Phi \rangle| &\leq |(\nabla \Psi^a_m(t), \nabla \Phi)| + |(V(t) \Psi^a_m(t),\Phi)| + |(G(t),\Phi)| \\
&\leq \|\nabla \Psi^a_m(t) \| + \|V\|_\infty \| \Psi^a_m(t) \| + \| G(t) \|,
\end{split}
\end{equation*}
for almost all $t \in (0,T)$, where $\langle \cdot, \cdot \rangle$ denotes the duality brackets.
Using \eqref{eq:EN_1} and \eqref{eq:EN_2},
we obtain that there exists a constant $K_{-1}$ independent of $m$ (but depending on $G$ and $\Psi_0$) such that
\begin{equation*}
\| \partial_t \Psi^a_m(t) \|_{H^{-1}} 
= \sup_{\Phi \in H^1_0(\Omega;\C) \, : \, \| \Phi \|_{H^1}\leq 1} |\langle \partial_t \Psi^a_m(t),\Phi \rangle |
\leq K_{-1}. %(\| G \|_{Y_{\infty,1}},\|\Psi_0\|_{H^1}).
\end{equation*}

Finally, let us prove \eqref{eq:EN_4}. Testing \eqref{eq:TDKS_weak_aux_m} with $\Phi=\partial_t \Psi^a_m(t)$, we obtain
\begin{equation*}
i \| \partial_t \Psi^a_m(t) \|^2 = (-\Delta \Psi^a_m(t) , \partial_t \Psi^a_m(t) ) + (V(t) \Psi^a_m(t) , \partial_t \Psi^a_m(t) ) + (G, \partial_t \Psi^a_m(t) ).
\end{equation*}
This implies
\begin{equation*}
\| \partial_t \Psi^a_m(t) \|^2 \leq \bigl( \| \Delta \Psi^a_m(t) \| + \| V \|_{\infty} \| \Psi^a_m(t) \| + \| G(t) \| \bigr) \| \partial_t \Psi^a_m(t) \|
\end{equation*}
and hence
\begin{equation*}
\| \partial_t \Psi^a_m(t) \| \leq  \| \Delta \Psi^a_m(t) \| + \| V \|_{\infty} \| \Psi^a_m(t) \| + \| G \|_{\Y}.
\end{equation*}
The result \eqref{eq:EN_4} follows using the other energy estimates.
\end{proof}

We are now ready to prove existence and uniqueness of a weak solution to the auxiliary problem \eqref{eq:TDKS_weak_aux}.

\begin{theorem}(Existence and uniqueness of a solution to the auxiliary problem)\label{thm:ex_un_sol}
For any $G \in \Y$ there exists a unique weak solution $\Psi \in W(0,T)$ to \eqref{eq:TDKS_weak_aux}
(with $\Psi(0) = \Psi_0$)
that satisfies the energy estimates of Theorem \ref{thm:energy_aux} and such that
\begin{equation*}
\Psi \in \Y, 
\quad 
\partial_t \Psi \in L^{\infty}(0,T;L^2(\Omega;\C))
\quad \text{and} \quad 
\Psi \in C([0,T];H^1_0(\Omega;\C)).
\end{equation*}
\end{theorem}

\begin{proof}
Consider a sequence of Galerkin approximations $\{ \Psi^a_m \}_m$. 
By Theorem \ref{thm:energy_aux} each function of this sequence lies in $W(0,T)$ and in particular $\Psi^a_m \in L^2(0,T;Z)$ and $\partial_t \Psi^a_m \in Y$.
The energy estimates of Theorem \ref{thm:energy_aux} together with (B3) guarantee that our Galerkin sequence is uniformly bounded in these spaces
by constants that depend only on the data of the problem. Since $W(0,T)$, $L^2(0,T;Z)$ and $Y$ are Hilbert spaces (hence reflexive),
there exists a subsequence $\{ \Psi^a_{m_j} \}_{j}$ that converges weakly in $W(0,T)$ and $L^2(0,T;Z)$ to a weak limit $\widehat{\Psi} \in W(0,T) \cap L^2(0,T;Z)$
with $\{ \partial_t \Psi^a_{m_j} \}_{j}$ converging weakly in $Y$ to $\partial_t \widehat{\Psi}$.
Moreover, the sequence $\{ \Psi^a_{m_j} \}_{j}$ converges strongly in $Y$ (by the compact embedding of $W(0,T)$ in $Y$; see, e.g., \cite{boyer2012mathematical,Lions1969}).
Using the linearity of operators and functionals in \eqref{eq:TDKS_weak_aux}, one can show that the limit $\widehat{\Psi}$ satisfies \eqref{eq:TDKS_weak_aux}
with $\widehat{\Psi}(0) = \Psi_0$; see also \cite{Evans2010} for similar arguments.
Hence, $\widehat{\Psi}$ is a weak solution to \eqref{eq:TDKS_weak_aux}.
Now, since for given positive constants $C_1$ and $C_2$ the sets
\begin{equation*}
\begin{split}
S_1 &:= \{ \Phi \in L^2(0,T;Z) \, : \, \| \Phi(t) \|_{H^2} \leq C_1 \text{ a.e. in $(0,T)$} \} \\
S_2 &:= \{ \Phi \in Y \, : \, \| \Phi(t) \| \leq C_2 \text{ a.e. in $(0,T)$} \} \\
\end{split}
\end{equation*}
are weakly closed, we have $\widehat{\Psi} \in S_1$ and $\partial_t \widehat{\Psi} \in S_2$.
With the same argument we obtain that $\widehat{\Psi}$ satisfies the energy estimates of
Theorem \ref{thm:energy_aux}.
Hence, $\widehat{\Psi} \in \Y$ and $\partial_t \widehat{\Psi} \in L^{\infty}(0,T;L^2(\Omega;\C))$.
Moreover, since $Z$ is compactly embedded in $H^1_0(\Omega;\C)$, the space
$W_\infty(0,T) := \{ \Phi \in \Y \, : \, \partial_t \Phi \in L^{\infty}(0,T;L^2(\Omega;\C)) \}$ 
is compactly embedded in $C([0,T];H^1_0(\Omega;\C))$; see \cite[Theorem II.5.16]{boyer2012mathematical}.
Therefore, $\widehat{\Psi} \in C([0,T];H^1_0(\Omega;\C))$.

To prove uniqueness, we proceed by contradiction and assume that there exists another function $\widetilde{\Psi} \in W(0,T)$,
distinct from $\widehat{\Psi}$, that solves \eqref{eq:TDKS_weak_aux}. If we define $\delta \Psi := \widehat{\Psi}-\widetilde{\Psi}$,
it is possible to show that this function satisfies the equation
\begin{equation}\label{eq:TDKS_weak_aux_err}
i (\partial_t \delta \Psi,\Phi) = (\nabla \delta \Psi, \nabla \Phi) + (V \delta \Psi,\Phi) \text{ a.e. in $(0,T)$, $\forall \Phi \in H^1_0(\Omega;\C)$}
\end{equation}
with $\delta \Psi(0) =0$. The energy estimate \eqref{eq:EN_1} remains valid for \eqref{eq:TDKS_weak_aux_err} and implies that $\| \delta \Psi \|_Y =0$,
which is a contradiction. Hence, the uniqueness of $\widehat{\Psi}$ follows.
\end{proof}

\section{Analysis of the full TDKS problem}\label{sec:full_NL}
In this section, we wish to show existence and uniqueness of a solution to \eqref{eq:TDKS_weak}.
We denote here the set $\Y$ by $\Y_T := L^\infty(0,T;Z)$ to stress the dependence on the final time $T$.

%To do so, consider three fixed positive constants $B_{A1}$, $B_{A2}$ and $B_{\circ}$, and a function $\Phi \in H^2(\Omega)$. We introduce the following conditions
%\begin{equation}\label{eq:COND_1}
%\| \Phi \|^2 \leq \exp(T) B_{A1} + \exp(T) \| \Psi_0\|^2,
%\end{equation}
%\begin{equation}\label{eq:COND_2}
%\| \nabla \Phi \|^2 \leq \exp(C_{\nabla} T) B_{A2} + \exp(C_{\nabla} T) C \| \nabla \Psi_0\|^2,
%\end{equation}
%where the constants $C_{\nabla}$, $C_{1,\Delta}$ and $C_{2,\Delta}$ are the same constants given in Theorem \ref{thm:energy_aux},
%and $C$ is the same constant given in (B3).
Consider a fixed positive constant $B_{\circ}$ and define
\begin{equation}\label{eq:COND_3}
C_{\circ}(T,B_{\circ}) := \exp(C_{1,\Delta} T) \Bigl[ B_{\circ} + \| \Delta \Psi_0 \|^2 + T C_{2,\Delta} \| \nabla \Psi_0 \|^2 \Bigr] ,
\end{equation}
where the constants $C_{1,\Delta}$ and $C_{2,\Delta}$ are the ones given in Theorem \ref{thm:energy_aux}.
For any $\widehat{T} \in (0,T]$ we define
\begin{equation}\label{eq:27}
S_{B_{\circ}}(\widehat{T}) := \{ \Phi \in \Y_{\widehat{T}} \, : \, \| \Phi(t) \|_Z \leq C_{\circ}(T,B_{\circ}) \text{ a.e. in $(0,\widehat{T})$} \}.
\end{equation}
Notice that, in view of the equivalence between the norms $\| \cdot \|_{H^2}$ and $\| \cdot \|_Z$ (see {\rm (B6)}),
the set $S_{B_{\circ}}(\widehat{T})$ is a ball in $\Y_{\widehat{T}}$ centered in zero and having radius depending
on $\widehat{T}$ and $B_{\circ}$.

Let $F$ be the nonlinear function given in \eqref{eq:F_NL}. 
Using {\rm (B5)}, {\rm (A5)}, \eqref{eq:27} and {\rm (B6)}, 
we have that for any $\Lambda \in S_{B_{\circ}}(T)$ the following estimates hold: 
\begin{equation}\label{eq:SSSS}
\begin{split}
\| F(\Lambda(t)) \|_{H^2} &= \| V_H(\Lambda(t)) \Lambda(t) + V_{xc}(\Lambda(t)) \Lambda(t) \|_{H^2} \\
&\leq C_b \| \Lambda(t) \|_{H^1}^2 \| \Lambda(t) \|_{H^2} + \widetilde{K} \| \Lambda(t) \|_{H^2} \\
&\leq C_b \| \Lambda(t) \|_{H^2}^3 + \widetilde{K} \| \Lambda(t) \|_{H^2} \\
&\leq C_b C_Z^3\| \Lambda(t) \|_{Z}^3 + \widetilde{K} C_Z \| \Lambda(t) \|_{Z} \\
&\leq C_b C_Z^3 C_{\circ}(T,B_{\circ})^3 + \widetilde{K} C_Z C_{\circ}(T,B_{\circ}),
\end{split}
\end{equation}
where $\widetilde{K}$ is given in {\rm (A5)}, $C_Z$ in {\rm (B6)} and $B_b$ in {\rm (B5)};
The estimate \eqref{eq:SSSS} implies that $F(\Lambda) \in \Y_{T}$.

Next, we introduce the map $\mathcal{A} : S_{B_{\circ}}(\widehat{T}) \rightarrow \Y_{\widehat{T}}$
defined by $\mathcal{A}(\Lambda) := \widetilde{\Psi}$, where $\widetilde{\Psi} \in \Y_{\widehat{T}}$
is the unique solution to
\begin{equation}\label{eq:TDKS_weak_aux_2}
i (\partial_t \widetilde{\Psi},\Phi) = (\nabla \widetilde{\Psi}, \nabla \Phi) + (V \widetilde{\Psi},\Phi) + (F(\Lambda),\Phi) \text{ a.e. in $(0,\widehat{T})$}
\end{equation}
for all $\Phi \in H^1_0(\Omega;\C)$ and with $\widetilde{\Psi}(0) = \Psi_0$.
Notice that, since $F(\Lambda) \in \Y_{\widehat{T}}$, this problem is
uniquely solvable in $\Y_{\widehat{T}}$ by Theorem \ref{thm:ex_un_sol}. Therefore, the map $\mathcal{A}$
is well defined. Let us prove the following property of $\mathcal{A}$.
\begin{lemma}\label{lemma:1}
For any $T^\star \in (0,T]$ such that
\begin{equation}\label{eq:cond_T}
T^\star \leq \frac{B_{\circ}}{C_{3,\Delta}\bigl[C_b C_Z^3 C_{\circ}(T,B_{\circ})^3 + \widetilde{K} C_Z C_{\circ}(T,B_{\circ})\bigr]^2}
\end{equation}
the set $S_{B_{\circ}}(T^\star)$ is invariant under $\mathcal{A}$, that is
$\mathcal{A}(\Lambda) \in S_{B_{\circ}}(T^\star)$ for any $\Lambda \in S_{B_{\circ}}(T^\star)$.
\end{lemma}

\begin{proof}
Take any $\Lambda \in \Y_{T^\star}$ and consider $\Psi = \mathcal{A}(\Lambda)$.
Since $F(\Lambda) \in \Y_{T^\star}$, we can use the energy estimates of Theorem \ref{thm:energy_aux}.
%Using \eqref{eq:EN_1}, \eqref{eq:SSSS} and \eqref{eq:cond_T}, we write
%\begin{equation*}
%\begin{split}
%\| \Psi(t) \|^2 &\leq \exp(T^*) \Bigl[ \| \Psi_0 \|^2 + T^\star \| F(\Lambda) \|_{\Y_{T^\star}}^2 \Bigr] \\
%&\leq \exp(T^*) \| \Psi_0 \|^2 +  \exp(T^*) T^\star \Bigl[ C_b B_{A1} \sqrt{B_{\circ}} + \widetilde{K} \sqrt{B_{\circ}} \Bigr]^2 \\
%&\leq \exp(T^*) \| \Psi_0 \|^2 +  \exp(T^*) B_{A1},
%\end{split}
%\end{equation*}
%for almost all $t$ in $(0,T^\star)$.
%
%Using \eqref{eq:EN_2}, \eqref{eq:SSSS} and \eqref{eq:cond_T}, we get
%\begin{equation*}
%\begin{split}
%\| \nabla \Psi(t) \|^2 &\leq \exp(C_\nabla T^\star) \Bigl[ C \| \nabla \Psi_0 \|^2 + T^\star \| F(\Lambda) \|_{\Y_{T^\star}}^2 \Bigr] \\
%&\leq \exp(C_\nabla T^\star) C \| \nabla \Psi_0 \|^2 + \exp(C_\nabla T^\star) T^\star \Bigl[ C_b B_{A1} \sqrt{B_{\circ}} + \widetilde{K} \sqrt{B_{\circ}} \Bigr]^2 \\
%&\leq \exp(C_\nabla T^\star) C \| \nabla \Psi_0 \|^2 + \exp(C_\nabla T^\star) B_{A2},
%\end{split}
%\end{equation*}
%for almost all $t$ in $(0,T^\star)$.
In particular, using \eqref{eq:EN_3}, \eqref{eq:SSSS} and \eqref{eq:cond_T}, we obtain
\begin{equation*}
\begin{split}
\| \Psi(t) \|_Z^2 &\leq \exp(C_{1,\Delta} T^\star) \Bigl[ \| \Delta \Psi_0 \|^2 + T^\star \bigl[ C_{2,\Delta} \| \nabla \Psi_0 \|^2 + C_{3,\Delta} \| F(\Lambda) \|_{\Y_{T^\star}}^2 \bigr] \Bigr] \\
&\leq \exp(C_{1,\Delta} T^\star) \Bigl[ \| \Delta \Psi_0 \|^2 + T^\star C_{2,\Delta} \| \nabla \Psi_0 \|^2 \Bigr] \\
&+  \exp(C_{1,\Delta} T^\star) T^\star C_{3,\Delta} \bigl[C_b C_Z^3 C_{\circ}(T,B_{\circ})^3 + \widetilde{K} C_Z C_{\circ}(T,B_{\circ})\bigr]^2 \\
&\leq \exp(C_{1,\Delta} T^\star) \Bigl[ \| \Delta \Psi_0 \|^2 + T^\star C_{2,\Delta} \| \nabla \Psi_0 \|^2 \Bigr] 
+ \exp(C_{1,\Delta} T^\star) B_{\circ} \\
&=C_{\circ}(T^\star,B_{\circ})
\end{split}
\end{equation*}
for almost all $t$ in $(0,T^\star)$, and the claim follows.
\end{proof}

The next lemma shows that the map 
$\mathcal{A} : S_{B_{\circ}}(\widehat{T}) \rightarrow S_{B_{\circ}}(\widehat{T})$
is a contraction for a sufficiently small $\widehat{T} \in (0,T^\star]$.

\begin{lemma}\label{lemma:contraction}
Consider $T^\star \in (0,T]$ such that \eqref{eq:cond_T} holds.
There exists a $T^{\star \star} \in (0,T^\star]$ such that for any 
$\widehat{T} \in (0,T^{\star \star}]$ the map
$\mathcal{A} : S_{B_{\circ}}(\widehat{T}) \rightarrow S_{B_{\circ}}(\widehat{T})$
is a contraction (with respect to the norm $\| \cdot \|_{\Y}$).
\end{lemma}

\begin{proof}
Consider any two functions $\Lambda_1$ and $\Lambda_2$ in $S_{B_{\circ}}(T^{\star})$
and define the function $\delta \Psi := \mathcal{A}(\Lambda_1) - \mathcal{A}(\Lambda_2)$.
This function solves the problem
\begin{equation}\label{eq:TDKS_weak_aux_3}
i (\partial_t \delta \Psi,\Phi) = (\nabla \delta \Psi, \nabla \Phi) + (V \delta \Psi,\Phi) + (F(\Lambda_1) - F(\Lambda_2),\Phi) \text{ a.e. in $(0,T^\star)$}
\end{equation}
for all $\Phi \in H^1_0(\Omega;\C)$ and $\delta \Psi(0) = 0$.
Since $F(\Lambda_1) - F(\Lambda_2) \in \Y_{T^\star}$, we can use the energy estimate \eqref{eq:EN_3}
to write
%\begin{equation}\label{eq:EN_11}
%\| \delta \Psi(t) \|^2 \leq \exp(T^\star) T^\star \| F(\Lambda_1) - F(\Lambda_2) \|_{\Y_{T^\star}}^2,
%\end{equation}
%\begin{equation}\label{eq:EN_22}
%\| \nabla \delta \Psi(t) \|^2 \leq \exp(C_\nabla T^\star) T^\star \| F(\Lambda_1) - F(\Lambda_2) \|_{\Y_{T^\star}}^2,
%\end{equation}
\begin{equation}\label{eq:EN_33}
\| \delta \Psi(t) \|_Z^2 \leq \exp(C_{1,\Delta} T^\star) T^\star C_{3,\Delta} \| F(\Lambda_1) - F(\Lambda_2) \|_{\Y_{T^\star}}^2.
\end{equation}
Let us estimate the term $\| F(\Lambda_1) - F(\Lambda_2) \|_{\Y_{T^\star}}^2$.
Using {\rm (B5)} and {\rm (A5)}, we get
\begin{equation*}
\begin{split}
\| F(\Lambda_1) - F(\Lambda_2) \|_{\Y_{T^\star}} &\leq \| f(\Lambda_1) - f(\Lambda_2) \|_{\Y_{T^\star}} + \| V_{xc}(\Lambda_1)\Lambda_1 - V_{xc}(\Lambda_2)\Lambda_2 \|_{\Y_{T^\star}} \\
&\leq C_c \bigl( \| \Lambda_1 \|_{\Y_{T^\star}}^2 + \| \Lambda_2 \|_{\Y_{T^\star}}^2 \bigr) \| \Lambda_1 - \Lambda_2 \|_{\Y_{T^\star}} + \widetilde{K} \| \Lambda_1 - \Lambda_2 \|_{\Y_{T^\star}}, \\
%&\leq \bigl[ C_c 2 (B_{A1} + B_{A2} + B_{\circ}) + \widetilde{K} \bigr] \| \Lambda_1 - \Lambda_2 \|_{\Y_{T^\star}},
\end{split}
\end{equation*}
where $\widetilde{K}$ is as in \eqref{eq:Lipschitz3} and $C_c$ as in \eqref{eq:11}. Since $\Lambda_1,\Lambda_2 \in S_{B_{\circ}}(T^{\star})$, we have
$\| \Lambda_j(t) \|_Z^2 \leq C_{\circ}(T^\star,B_{\circ})$, for $j=1,2$.
Therefore, one obtains
\begin{equation}\label{eq:est_FF}
\| F(\Lambda_1) - F(\Lambda_2) \|_{\Y_{T^\star}} \leq C_{\circ \circ}( T^\star ) \| \Lambda_1 - \Lambda_2 \|_{\Y_{T^\star}},
\end{equation}
where $C_{\circ \circ}(T^\star) = \widetilde{K} + 2 C_c C_{\circ}(T^\star,B_{\circ})$.
Using \eqref{eq:est_FF} into \eqref{eq:EN_33}, we get
\begin{equation}
\| \delta \Psi(t) \|_Z^2 \leq C_{\circ \circ}(T^\star)^2 \exp(C_{1,\Delta} T^\star) T^\star C_{3,\Delta} \| \Lambda_1 - \Lambda_2 \|_{\Y_{T^\star}}^2.
\end{equation}
This implies that
\begin{equation*}
\| \mathcal{A}(\Lambda_1) - \mathcal{A}(\Lambda_2) \|_{\widehat{Y}_{T^\star}} = \| \delta \Psi \|_{\widehat{Y}_{T^\star}} \leq g(T^\star) \| \Lambda_1 - \Lambda_2 \|_{\Y_{T^\star}},
\end{equation*}
where
\begin{equation*}
g(T^\star) = C_Z C_{\circ \circ}(T^\star) \sqrt{\exp(C_{1,\Delta} T^\star) T^\star C_{3,\Delta}}.
\end{equation*}
Since the map $\widehat{T} \mapsto g(\widehat{T})$ is continuous and monotonically increasing in $[0,T^\star]$ 
with $g(0)=0$, there exists a $T^{\star \star} \in (0,T^\star]$
such that $g(T^{\star \star}) < 1$, and we obtain
\begin{equation*}
\| \mathcal{A}(\Lambda_1) - \mathcal{A}(\Lambda_2) \|_{\widehat{Y}_{T^{\star \star}}} = \| \delta \Psi \|_{\widehat{Y}_{T^\star}} \leq g(T^\star) \| \Lambda_1 - \Lambda_2 \|_{\Y_{T^{\star \star}}},
\end{equation*}
which shows that the mapping $\mathcal{A} : S_{B_{\circ}}(T^{\star \star}) \rightarrow S_{B_{\circ}}(T^{\star \star})$ is a contraction.
\end{proof}

\begin{remark}
The results of Lemmas \ref{lemma:1} and \ref{lemma:contraction}
still hold if one assumes \eqref{eq:Lipschitz3} to hold only on bounded subsets of $Z$.
\end{remark}

We are now ready to show existence of a unique solution to the TDKS problem \eqref{eq:TDKS_weak}.

\begin{theorem}[Existence and uniqueness of solution to the TDKS problem]\label{thm:existence_uniqueness}
There exists a unique solution $\Psi \in W(0,T)$ to \eqref{eq:TDKS_weak} with $\Psi(0)=\Psi_0$.
In particular, it holds that
\begin{equation*}
\Psi \in \Y, 
\quad 
\partial_t \Psi \in L^{\infty}(0,T;L^2(\Omega;\C))
\; \text{ and} \quad 
\Psi \in C([0,T];H^1_0(\Omega;\C)).
\end{equation*}
\end{theorem}

\begin{proof}
We proceed in 4 steps.

\underline{Step 1}: Existence and uniqueness for sufficiently short time interval.\\
Consider the problem \eqref{eq:TDKS_weak} defined on a time interval $[0,T_1^{\star\star}]$,
where $T_1^{\star\star} \in (0,T^\star]$ is chosen sufficiently small as in Lemma \ref{lemma:contraction}. 
By Lemma \ref{lemma:contraction} the mapping 
$\mathcal{A} : S_{B_{\circ}}(T_1^{\star\star}) \rightarrow S_{B_{\circ}}(T_1^{\star\star})$ is a contraction.
Therefore, the Banach-Caccioppoli fixed-point theorem \cite{Ciarlet2013,Evans2010} implies that there exists
a unique solution $\widehat{\Psi}_1 \in S_{B_{\circ}}(T_1^{\star\star})$ to the problem \eqref{eq:TDKS_weak}.

\underline{Step 2}: Existence for the entire time interval $[0,T]$.\\
We follow the same procedure used in \cite[Section 9.2.1]{Evans2010} to construct a unique
solution to \eqref{eq:TDKS_weak} on $[0,T]$.
Since $\widehat{\Psi}_1(t) \in Z$ for almost all $t$ in $[0,T_1^{\star\star}]$, we can assume that
$\widehat{\Psi}_1(T_1^{\star\star}) \in Z$ (upon redefining $T_1^{\star\star}$ if necessary).

Now, we define a new problem which is identical to \eqref{eq:TDKS_weak}, but defined on the time interval $[T_1^{\star \star},T]$ and having initial condition 
$\widehat{\Psi}_2(T_1^{\star \star})=\widehat{\Psi}_1(T_1^{\star \star})$.
By repeating the same arguments as in Step 1, one can show that there exists a $T_2^{\star \star} \in (T_1^{\star \star},T]$ 
such that the new problem is uniquely solved on $[T_1^{\star \star},T_2^{\star \star}]$ by $\widehat{\Psi}_2$ which lies
in a ball similar to $S_{B_{\circ}}(T_1^{\star\star})$ whose radius depends on the length of the interval $[T_1^{\star \star},T_2^{\star \star}]$.
Since $\widehat{\Psi}_2(t) \in Z$ for almost all $t$ in $[T_1^{\star \star},T_2^{\star \star}]$, we can assume that
$\widehat{\Psi}_2(T_2^{\star\star}) \in Z$ (upon redefining $T_2^{\star\star}$ if necessary).

This argument can be repeated recursively. At the $k^{th}$ recursion step we obtain the existence
of the solution $\widehat{\Psi}_k$ to \eqref{eq:TDKS_weak} on $[T_{k-1}^{\star \star},T_k^{\star \star}]$ with initial condition $\widehat{\Psi}_k(T_{k-1}^{\star \star})=\widehat{\Psi}_{k-1}(T_{k-1}^{\star \star})$.
In Step 3 we prove that a finite number of these intervals $[T^{**}_{k-1}, T^{**}_k]$ are sufficient to cover the whole interval $[0,T]$, that is
$\cup_{k=1}^M [T_{k-1}^{\star \star},T_k^{\star \star}] = [0,T]$, with $T_0^{\star \star}=0$,
a positive finite integer $M$,
and a solution $\Psi \in \Y$ to \eqref{eq:TDKS_weak} on $[0,T]$.
%The fact that only finitely many recursion steps are needed is proved in Step 3.

\underline{Step 3}: Existence of a finite covering of $[0,T]$.\\
In Step 2 we created a sequence $\{T_k^{\star \star}\}_k$ and claimed that there exists a finite positive integer $M$ such that $\cup_{k=1}^M [T_{k-1}^{\star \star},T_k^{\star \star}] = [0,T]$.
To prove this result, we look at the differences $T^d_k := T_k^{\star \star} - T_{k-1}^{\star \star}$ for $k=1,2,\dots$ and notice
that choosing $T_k^{\star \star}$ is equivalent to choose $T^d_k$. In fact, at the $k^{th}$ recursion step the solution $\widehat{\Psi}_k$
is defined on $[T_{k-1}^{\star \star},T_k^{\star \star}]$, and a simple time shift allows us to pose the corresponding problem 
on $[0,T^d_k]$ (instead of $[T_{k-1}^{\star \star},T_k^{\star \star}]$). 
We denote by $\widetilde{\Psi}_k$ the (shifted) solution defined on $[0,T^d_k]$
and notice that $\widetilde{\Psi}_k(t) = \widehat{\Psi}_k(T_{k-1}^{\star \star}+t)$.
Therefore, to prove the existence of a finite $M$ it suffices
to construct a sequence $\{T^d_k\}_k$ with $\sum_{k=1}^\infty T^d_k = \infty$ and such that each $T^d_k$ is sufficiently small
in the sense of Lemmas \ref{lemma:1} and \ref{lemma:contraction}. 
To be more precise, we will construct a sequence $\{T_k^d\}_k$ that converges
to zero ``sufficiently slowly'' and then choose $B_{\circ,k}=O(T_k^d)$.
Moreover, our choices of $\{T_k^d\}_k$ and $\{ B_{\circ,k} \}_k$ will satisfy
\begin{equation}\label{eq:cond_T_1}
T^d_k \leq \frac{B_{\circ,k}}{C_{3,\Delta}\bigl[C_b C_Z^3 C_{\circ,k}(T,B_{\circ,k})^3 + \widetilde{K} C_Z C_{\circ,k}(T,B_{\circ,k})\bigr]^2},
\end{equation}
and 
\begin{equation}\label{eq:cond_T_2}
g_k(T^d_k) < 1,
\end{equation}
where
\begin{equation}\label{eq:COND_3_k}
\begin{split}
C_{\circ,k}(T,B_{\circ,k}) &:= \exp(C_{1,\Delta} T) \Bigl[ B_{\circ,k} + a_k (1 + T C_{2,\Delta} ) \Bigr], \\
g_k(T^d_k) &:= C_Z C_{\circ \circ,k}(T^d_k) \sqrt{\exp(C_{1,\Delta} T^d_k) T^d_k C_{3,\Delta}},\\
C_{\circ \circ,k}(T^d_k) &:= \widetilde{K} + 2 C_c C_{\circ,k}(T^d_k,B_{\circ,k}), \\
\end{split}
\end{equation}
where we defined $a_k:=\| \widetilde{\Psi}_k(T_k^d) \|_Z^2$. 
Notice that the conditions \eqref{eq:cond_T_1} and \eqref{eq:cond_T_2} are exactly the ones considered in Lemmas \ref{lemma:1} and \ref{lemma:contraction}.

Following the proof of Theorem \ref{thm:energy_aux}, setting $G=F(\widehat{\Psi}_k)$ and using \eqref{eq:SSSS}, one gets
\begin{equation}\label{eq:EN_NEWW}
\begin{split}
a_k &\leq \exp(C_{1,\Delta} T^d_k) \Bigl[ a_{k-1} ( 1 +  T^d_k C_{2,\Delta} ) \\
&+ T^d_k C_{3,\Delta} \Bigl( C_b C_Z^3 C_{\circ,k}(T^d_k,B_{\circ,k})^3 + \widetilde{K} C_Z C_{\circ,k}(T^d_k,B_{\circ,k}) \Bigr) \Bigr].
\end{split}
\end{equation}
To ease the notation, we define $C_{21}:=C_{2,\Delta}/C_{1,\Delta}$, $C_{31}:=C_{3,\Delta}/C_{1,\Delta}$, and $T^d_k := \frac{b_k}{C_{1,\Delta}}$,
where $\{b_k\}_k$ is a sequence to be defined. These definitions allow us to rewrite \eqref{eq:EN_NEWW} as
\begin{equation}\label{eq:EN_NEWW_2}
a_k \leq f(a_{k-1},b_k,B_{\circ,k}),
\end{equation}
where
\begin{equation*}
\begin{split}
f(a_{k-1},b_k,B_{\circ,k}) &:= \exp(b_k) \Bigl[ a_{k-1} ( 1 +  C_{21} b_k ) \\
&+ b_k C_{31} \Bigl( C_b C_Z^3 C_{\circ,k}\bigl(\frac{b_k}{C_{1,\Delta}},B_{\circ,k}\bigr)^3 + \widetilde{K} C_Z C_{\circ,k}\bigl(\frac{b_k}{C_{1,\Delta}},B_{\circ,k}\bigr) \Bigr) \Bigr].
\end{split}
\end{equation*}
A direct inspection reveals that the map $B_{\circ,k} \mapsto f(a_{k-1},b_k,B_{\circ,k})$ is 
continuous and monotonically increasing.

Since we aim at constructing a convergent-to-zero sequence $\{T_k^d\}_k$ and
$B_{\circ,k} = O(T_k^d)$, we make the following ansatz: 
the sequence $\{ B_{\circ,k} \}_k$ is bounded by a positive global constant $B$. 
Since $f$ is monotonically increasing in $B_{\circ,k}$, 
the boundedness of $\{ B_{\circ,k} \}_k$ implies that
$f(a_{k-1},b_k,B_{\circ,k}) \leq f(a_{k-1},b_k,B)$.

Let us consider the two sequences $\{a_k\}_k$ and $\{ b_k \}_k$.
Without loss of generality, we consider the case in which $\{a_k\}_k$ is possibly growing and
$\{b_k\}_k$ converges to zero. Our goal is to construct a sequence $\{b_k\}_k$ that converges to
zero ``slowly enough'', while $\{a_k\}_k$ grows ``slowly enough''. This is necessary to satisfy \eqref{eq:cond_T_1},
where the left-hand converges to zero, while the right-hand side could grow as $k\rightarrow \infty$.
We then assume that $a_k > 1$ and $b_k < 1$ for $k$ large enough.

A Taylor expansion of $b \mapsto f(a,b,B)$ in $b=0$ reveals that
\begin{equation*}
f(a,b,B) = a + b(\widehat{A}+\widehat{B}a+\widehat{C}a^2+\widehat{D}a^3) + O(b^2a^3),
\end{equation*}
for some positive constants $\widehat{A},\widehat{B},\widehat{C},\widehat{D}$.
Therefore, \eqref{eq:EN_NEWW_2} becomes
\begin{equation}\label{eq:EN_NEWW_3}
a_k \leq a_{k-1} + b_k(\widehat{A}+\widehat{B}a_{k-1}+\widehat{C}a_{k-1}^2+\widehat{D}a_{k-1}^3) + O(b_k^2 a_{k-1}^3).
\end{equation}
Now, we choose $b_k = \frac{1}{k a_{k-1}^3}$, which clearly satisfies $b_k<1$ and $b_k \rightarrow 0$,
and obtain
\begin{equation}\label{eq:EN_NEWW_4}
\begin{split}
a_k &\leq a_{k-1} \Bigl( 1 + \frac{\widehat{A}}{k a_{k-1}^4}+\frac{\widehat{B}}{k a_{k-1}^3}+\frac{\widehat{C}}{k a_{k-1}^2}+\frac{\widehat{D}}{k a_{k-1}} \Bigr) + O\bigl(\frac{1}{k^2 a_{k-1}^3}\bigr) \\
&\leq a_{k-1} \Bigl( 1 + \frac{\widehat{E}}{k} \Bigr) \leq a_0 \Bigl( 1 + \frac{\widehat{E}}{k} \Bigr)^k \leq a_0 \exp( \widehat{E} ),
\end{split}
\end{equation}
for some positive constant $\widehat{E}$, where we used that $a_{k-1} \geq 1$.
The estimate \eqref{eq:EN_NEWW_4} shows that the sequence $\{a_k\}_k$ is bounded if we choose $\{b_k\}_k$ as $b_k = \frac{1}{k a_{k-1}^3}$.
This means that the sequence $a_k = \| \widetilde{\Psi}_k(T^d_k) \|_Z^2$ is bounded if one chooses $T^d_k = \frac{b_k}{C_{1,\Delta}} = \frac{1}{k C_{1,\Delta} \| \widetilde{\Psi}_{k-1}(T^d_{k-1}) \|_Z^6 }$.
Therefore, $b_k \rightarrow 0$ and $T^d_k\rightarrow 0$ as $k \rightarrow \infty$, and $\sum_{k=1}^{\infty} T^d_k = \infty$.
We found a suitable candidate $\{T^d_k\}_k$ for our purposes.

It remains to show the existence of a sequence $\{ B _{\circ,k} \}_k$ that is
bounded (according to our ansatz) and that satisfies \eqref{eq:cond_T_1} and \eqref{eq:cond_T_2}.
Let us look at \eqref{eq:cond_T_1}. Since the sequence $\{ \| \widetilde{\Psi}_k(T^d_k) \|_Z^2 \}_k$ is bounded, the denominator of the right-hand side of \eqref{eq:cond_T_1} is bounded.
Therefore, it is possible to find a $B_{\circ,k} = O(T^d_k)$ such that \eqref{eq:cond_T_1} is satisfied for any $k$.

Let us consider \eqref{eq:cond_T_2}. A direct calculation allows us to obtain that
\begin{equation*}
g_k(T^d_k) \leq \widehat{g}(T^d_k),
\end{equation*}
where
\begin{equation*}
\begin{split}
\widehat{g}(T^d_k) &= C_Z \Bigl\{ \widetilde{K} + 2 C_c \exp( C_{1,\Delta} T^d_k ) \times \\
&\times \bigl[ B + a_0 \exp( \widehat{E} ) (1+C_{2,\Delta} T^d_k) \bigr]\Bigr\}
\sqrt{\exp( C_{1,\Delta} T^d_k ) T^d_k C_{3,\Delta}},
\end{split}
\end{equation*}
which is continuous, monotonically increasing in $T^d_k$ and $\widehat{g}(0)=0$.
Therefore, it is sufficient to redefine $T^d_k$ by dividing it by an appropriate positive constant to obtain that \eqref{eq:cond_T_2} holds.

We have created a sequence $\{T^d_k\}_k$ whose elements satisfy the conditions of Lemmas \ref{lemma:1} and \ref{lemma:contraction} and such that
the corresponding series diverges. Therefore, there exists a finite positive integer $M$ such that $\sum_{k=1}^M T^d_k = T$.

\underline{Step 4}: Uniqueness and regularity of the constructed solution.\\
Uniqueness of such solution follows by the same arguments used in Theorem~\ref{thm:ex_un_sol},
but with \eqref{eq:EN_1} replaced by a similar energy estimate
obtained using Gr\"onwall's inequality
and the Lipschitz continuity conditions \eqref{eq:Lipschitz1} and \eqref{eq:Lipschitz2}.

The regularity of $\Psi$ follows from the fact that, by Theorem~\ref{thm:ex_un_sol} and
\eqref{eq:SSSS}, on each subinterval of the finite covering of $[0,T]$ the solution
$\Psi$ satisfies energy estimates similar to the ones of Theorem~\ref{thm:energy_aux}.
\end{proof}

\section{Less regular initial condition and potentials}\label{sec:less_regular}
In the previous sections, we assumed {\rm (A2)}, {\rm (A4)} and {\rm (A5)}, which essentially 
require the potentials $V_0,V_u,V_{xc}$ and the initial condition function $\Psi_0$ to be
twice weakly differentiable.
Is it possible to exploit our results to obtain
existence and uniqueness of a solution also in case of less regular potentials?

If the initial condition function $\Psi_0$ is less regular, then
it is natural to expect less regularity of the corresponding solution. 
In general, a Schr\"odinger equation ``transports'' in time the regularity
of the initial condition. Some maximal regularity results are proved in \cite{cazenave2003semilinear},
where the unique solution has the same space regularity
of the initial condition. An extreme case is also shown in \cite{Rugg2015}, where the authors prove
that, for a simple Schr\"odinger equation defined on a bounded one-dimensional space domain, an initial
condition in $L^2$ (outside the domain of the Hamiltonian) leads to a solution (so-called `mild solution') that is 
continuous but nowhere differentiable in time and continuous but nowhere differentiable in space
for almost all times. For these reasons, in this section we prove existence
and regularity results in case of a less regular initial condition $\Psi_0 \in H^1_0(\Omega;\C)$.

The proofs of the results presented in the previous sections
strongly rely on the regularity assumption on the potentials $V_0$, $V_u$ and $V_{xc}$
being twice weakly differentiable. 
In this section, we prove existence and regularity results for our problem \eqref{eq:TDKS_weak}
in the case when weakly differentiable potentials $W := W_0 + W_u w$ and $W_{xc}$ are added to the
right-hand side of our equation, i.e.
\begin{equation}\label{eq:TDKS_weak_less_reg}
i (\partial_t \Psi,\Phi) = (\nabla \Psi, \nabla \Phi) + ((V+W)\Psi + W_{xc}(\Psi),\Phi) + (F(\Psi),\Phi) \text{ a.e. in $(0,T)$},
\end{equation}
for all $\Phi \in H^1_0(\Omega;\C)$.

More precisely, we consider the following new assumptions.
\begin{itemize}
\item[{\rm (A2b)}]\label{ass:A2b} $W_0,W_u \in W^{1,\infty}(\Omega;\R)$.
\item[{\rm (A3b)}]\label{ass:A3b} $w \in L^{\infty}(0,T;\R)$.
\item[{\rm (A4b)}]\label{ass:A4b} $\Psi_0 \in H^1_0(\Omega;\C)$.
\item[{\rm (A5b)}]\label{ass:A5b} For any $\Phi \in H^1_0(\Omega;\C)$ it
holds that $W_{xc}(\Phi)\Phi \in H^1_0(\Omega;\C)$ and there exist positive constants $K_1$ and $K_2$
such that
\begin{equation}\label{eq:Lipschitz2_new}
\| W_{xc}(\Phi)\Phi - W_{xc}(\Lambda)\Lambda \| \leq K_1 \| \Phi - \Lambda \|
\end{equation}
\begin{equation}\label{eq:Lipschitz3_new}
\| W_{xc}(\Phi)\Phi - W_{xc}(\Lambda)\Lambda \|_{H^1} \leq K_2 \| \Phi - \Lambda \|_{H^1}
\end{equation}
for any $\Phi,\Lambda \in H^1_0(\Omega;\C)$.
\end{itemize} 
Because of the above discussion, we expect existence of a solution which is 
a function in $H^1_0(\Omega;\C)$ almost everywhere in $(0,T)$.
To show this, we will need the following technical lemma.

\begin{lemma}\label{lemma:Hartree}
There exists a positive constant $\widehat{C}_H$ such that
\begin{equation}\label{eq:Hartree_Lipschitz}
\| \nabla ( V_H(\Phi) \Phi ) -  \nabla ( V_H(\Lambda) \Lambda ) \| 
\leq \widehat{C}_H \bigl( \| \Phi \|_{H^1}^2 + \| \Lambda \|_{H^1}^2 \bigr) \| \Phi - \Lambda \|_{H^1}
\end{equation}
for any $\Phi,\Lambda \in H^1(\Omega;\C)$.
\end{lemma}

\begin{proof}
We prove the statement for any $\Phi,\Lambda \in H^1(\R^3;\C)$. Once this is achieved, the estimate
\eqref{eq:Hartree_Lipschitz} follows easily by considering $\mathcal{E}(\Phi)$ and $\mathcal{E}(\Lambda)$ for any $\Phi,\Lambda \in H^1(\Omega;\C)$
and using the continuity of the extension operator $\mathcal{E}$.
In this proof, $(\cdot,\cdot)$ denotes the inner product for $L^2(\R^3;\C)$ and $\| \cdot \|$ the corresponding
norm. Similarly, $\| \cdot \|_{H^1}$ and $\| \cdot \|_{L^\infty}$ denote the usual norms for 
$H^1(\R^3;\C)$ and $L^\infty(\R^3;\C)$. 

Let $z(x):=\frac{1}{|x|}$. For any three functions $a,b,c \in H^1(\R^3)$, we write
\begin{equation}\label{eq:formula_abc}
\nabla[ ((ab) \star z) c ] = ((ab)\star z) \nabla c + (((\nabla a) b) \star z)c + ((a \nabla b) \star z)c.
\end{equation}
Using Cauchy-Schwarz and Hardy's inequalities
(see, e.g., \cite[Lemma 4.1]{yserentant2010regularity}
and \cite{CancesLeBris1999}), we obtain
\begin{equation}\label{eq:formula_abc_2}
\begin{split}
| ((ab)\star z)(x) | 
&= \Bigl| ( a , \frac{b}{|x - \cdot|} ) \Bigr| \leq \| a \| \Bigl\| \frac{b}{|x - \cdot|} \Bigr\|
\leq C_H  \| a \| \| \nabla b \|, \\
| (((\nabla a) b)\star z)(x) | 
&= \Bigl| ( \nabla a , \frac{b}{|x - \cdot|} ) \Bigr| \leq \| \nabla a \| \Bigl\| \frac{b}{|x - \cdot|} \Bigr\|
\leq C_H  \| \nabla a \| \| \nabla b \|, \\
| ((a \nabla b)\star z)(x) | &= \Bigl| ( \nabla b , \frac{a}{|x - \cdot|} ) \Bigr| 
\leq \| \nabla b \| \Bigl\| \frac{a}{|x - \cdot|} \Bigr\|
\leq C_H  \| \nabla a \| \| \nabla b \|, \\
\end{split}
\end{equation}
for almost all $x \in \R^3$.

The triangle inequality yields the estimate
\begin{equation}\label{eq:0000}
\begin{split}
\| \nabla (V_H(\Phi) \Phi) - \nabla (V_H(\Lambda) \Lambda) \|
\leq \| \nabla (V_H(\Phi)(\Phi-\Lambda)) \| + \| \nabla ((V_H(\Phi)-V_H(\Lambda)) \Lambda) \|,
\end{split}
\end{equation}
where the two terms on the right-hand side must be suitably bounded.

Consider the term $\| \nabla (V_H(\Phi)(\Phi-\Lambda)) \|$.
The relations \eqref{eq:formula_abc} and \eqref{eq:formula_abc_2} allow us to compute
\begin{equation}\label{eq:1111}
\begin{split}
\| \nabla (V_H(\Phi)(\Phi-\Lambda)) \| &= \| \nabla ((\Phi \overline{\Phi}) \star z) (\Phi-\Lambda)) \| \\
&\leq \| \nabla ((\Phi \overline{\Phi}) \star z) \|_{L^{\infty}} \| \Phi-\Lambda \|
+ \| (\Phi \overline{\Phi}) \star z \|_{L^{\infty}} \| \nabla \Phi- \nabla \Lambda \| \\
&\leq C_I(C_H) \| \Phi \|_{H^1}^2 \| \Phi-\Lambda \|_{H^1},
\end{split}
\end{equation}
for some constant $C_I$ depending on the Hardy's inequality constant $C_H$.

Consider the term $\| \nabla ((V_H(\Phi)-V_H(\Lambda)) \Lambda) \|$.
Notice that $V_H(\Phi)-V_H(\Lambda)$ can be written as 
$$V_H(\Phi)-V_H(\Lambda) = (\Phi \overline{(\Phi-\Lambda)}) \star z + (\overline{\Lambda}(\Phi-\Lambda)) \star z,$$
which then implies that
\begin{equation*}
\| \nabla ((V_H(\Phi)-V_H(\Lambda)) \Lambda) \| 
\leq 
\| \nabla \Bigl[ ((\Phi \overline{(\Phi-\Lambda)}) \star z) \Lambda \Bigr] \| 
+ \| \nabla \Bigl[ ((\overline{\Lambda}(\Phi-\Lambda)) \star z)  \Lambda \Bigr] \|.
\end{equation*}
Using again \eqref{eq:formula_abc} and \eqref{eq:formula_abc_2}, a direct calculation 
(similar to \eqref{eq:1111}) allows us to obtain
\begin{equation}\label{eq:2222}
\| \nabla ((V_H(\Phi)-V_H(\Lambda)) \Lambda) \| 
\leq
C_{II}(C_H) ( \| \Phi \|_{H^1}^2 + \| \Lambda \|_{H^1}^2 ) \| \Phi-\Lambda \|_{H^1},
\end{equation}
where $C_{II}$ is a positive constant depending on $C_H$.
The result \eqref{eq:Hartree_Lipschitz} follows by \eqref{eq:0000}, \eqref{eq:1111}, \eqref{eq:2222},
and setting $\widehat{C}_H := \max\{ C_I(C_H), C_{II}(C_H)\}$.
\end{proof}

We exploit the results obtained in Sections \ref{sec:auxiliary} and \ref{sec:full_NL}
to prove existence and uniqueness of a solution to \eqref{eq:TDKS_weak_less_reg}. 
Let us introduce the following smooth approximations; see, e.g., \cite{lieb2001analysis,Jerome2019}.
Suppose that a non-negative smooth function $\phi$ with compact support ($\phi \in C^\infty_c(\R^3)$) and such that $\int_{\R^3} \phi = 1$ is given. Define
$$\phi_\epsilon(x) := \epsilon^{-3} \phi(x/\epsilon),$$
so that $\int_{\R^3} \phi_\epsilon = 1$ and $\| \phi_\epsilon \|_{L^1} = \| \phi \|_{L^1}$; see, e.g., \cite{lieb2001analysis}.
Consider the smoothed initial condition
\begin{equation*}
\Psi_{0,\epsilon} = \phi_\epsilon \star \mathcal{E} (\Psi_0)
\end{equation*}
and the smoothed potentials
\begin{equation*}
W_{\epsilon}(\Phi) = \phi_\epsilon \star \mathcal{E}(W \Phi)
\quad \text{ and} \quad
W_{xc,\epsilon}(\Phi) = \phi_\epsilon \star \mathcal{E}(W_{xc}(\Phi)\Phi),
\end{equation*}
for any $\Phi \in Z$, where $\mathcal{E}$ is the extension operator introduced in Section \ref{sec:formulation}.
Clearly, the above convolutions are restricted on $\Omega$
and classical results \cite[Section 2.16]{lieb2001analysis} guarantee that the corresponding approximation functions are smooth,
that is $\Psi_{0,\epsilon}$, $W_{\epsilon}(\Phi)$ and $W_{xc,\epsilon}(\Phi)$ are in $C^\infty(\Omega;\C)$ with
\begin{equation*}
\lim_{\epsilon \rightarrow 0} \Psi_{\epsilon,0} = \Psi_0,
\quad
\lim_{\epsilon \rightarrow 0} W_{\epsilon}(\Phi) = W(\Phi)
\quad \text{and} \quad 
\lim_{\epsilon \rightarrow 0} W_{xc,\epsilon}(\Phi) = W_{xc}(\Phi)\Phi
\end{equation*}
in $L^2(\Omega)$. We define
$V_{\epsilon}(\Psi) := V \Psi + W_{\epsilon}(\Psi) + W_{xc,\epsilon}(\Psi)$
and consider the problem
\begin{equation}\label{eq:TDKS_weak_epsilon}
i (\partial_t \Psi,\Phi) = (\nabla \Psi, \nabla \Phi) + (V_{\epsilon}(\Psi),\Phi) + (F(\Psi),\Phi) \text{ a.e. in $(0,T)$}
\end{equation}
for all $\Phi \in H^1_0(\Omega;\C)$ and $\Psi(0)=\Psi_{0,\epsilon}$.

To use the results of Sections \ref{sec:auxiliary} and \ref{sec:full_NL},
we study the auxiliary Galerkin problem
\begin{equation}\label{eq:TDKS_weak_epsilon_m}
i (\partial_t \Psi^a_m,\Phi) = (\nabla \Psi^a_m, \nabla \Phi) + (V_{\epsilon}(\Psi^a_m),\Phi) + (G,\Phi) \text{ a.e. in $(0,T)$}
\end{equation}
for all $\Phi \in W_m$, which is exactly \eqref{eq:TDKS_weak_aux_m} with $V \Psi$ replaced by $V_{\epsilon}(\Psi)$
and $\Psi_0$ replaced by $\Psi_{0,\epsilon}$. 
Notice that the smoothness of the regularized potential $V_\epsilon$ guarantees that \eqref{eq:TDKS_weak_epsilon_m}
is uniquely solvable by a solution having the same regularity of the solution to \eqref{eq:TDKS_weak_aux_m}.

The first key step is to prove energy estimates for the solution to \eqref{eq:TDKS_weak_epsilon_m}
and study the dependence of the obtained bounds on the regularization parameter $\epsilon$.
For this purpose, we begin with some preliminary lemmas.

\begin{lemma}\label{lemma:reg_in_cond}
Consider the initial condition function $\Psi_0$ and its regularization $\Psi_{0,\epsilon}$.
There exist positive constants $C_0$, $C_1$ and $C_2$ (independent of $\epsilon$)
such that for any $\epsilon >0$ it holds that
\begin{equation}\label{eq:est_in_cond}
\begin{split}
\| \Psi_{0,\epsilon} \| &\leq C_0 \| \phi_1 \|_{L^1(\R^3)} \| \Psi_0 \|, \\
\| \nabla \Psi_{0,\epsilon} \| &\leq C_1 \| \phi_1 \|_{L^1(\R^3)} \| \Psi_0 \|_{H^1}, \\
\| \Delta \Psi_{0,\epsilon} \| &\leq C_2 \| \nabla \phi_\epsilon \|_{L^1(\R^3)} \| \Psi_0 \|_{H^1},
\end{split}
\end{equation}
where $\phi_1 = \phi_{\epsilon=1}$.
\end{lemma}
\begin{proof}
The first estimate follows by a standard result
for approximation by convolution (see, e.g., \cite[Section 2.16]{lieb2001analysis}) and by the continuity of the extension operator.

The second estimate is obtained by exploiting the formula for the differentiation of convolution
and Young's inequality.
Indeed, we have $\nabla \Psi_{0,\epsilon} = \nabla( \phi_\epsilon \star \mathcal{E}(\Psi_0) )
= \phi_\epsilon \star \nabla \mathcal{E}(\Psi_0)$ and hence
\begin{equation*}
\begin{split}
\| \nabla \Psi_{0,\epsilon} \| &= \| \phi_\epsilon \star \nabla \mathcal{E}(\Psi_0) \| 
\leq \| \phi_1 \|_{L^1(\R^3)} \| \nabla \mathcal{E}(\Psi_0) \|_{L^2(\R^3)} \\
&\leq \| \phi_1 \|_{L^1(\R^3)} \| \mathcal{E}(\Psi_0) \|_{H^1(\R^3)}
\leq C_{\mathcal{E}} \| \phi_1 \|_{L^1(\R^3)} \| \Psi_0 \|_{H^1},
\end{split}
\end{equation*}
where $C_{\mathcal{E}}$ is the continuity constant of the operator $\mathcal{E}$;
see {\rm (B4)}.

To obtain the third inequality, we do not differentiate twice $\Psi_0$ (since it is only 
in $H^1_0(\Omega;\C)$),
but consider the distribution induced by $\nabla \mathcal{E}(\Psi_0)$ acting on the smooth
function $\phi_{\epsilon}$.
We then write
\begin{equation*}
\partial_{x_j} (\phi_\epsilon \star \partial_{x_j}  \mathcal{E}(\Psi_0))
= \int_{\R^3} \partial_{x_j} \phi_\epsilon(x-y) \partial_{x_j}  \mathcal{E}(\Psi_0)(y) \, dy,
\end{equation*}
for $j=1,2,3$,
which is the derivative of the distribution induced by $\nabla \mathcal{E}(\Psi_0)$ and applied to $\Phi_\epsilon$.
This formula allows us to compute
\begin{equation*}
\| \Delta \Psi_{0,\epsilon} \| = \| \nabla \cdot (\phi_\epsilon \star \nabla \mathcal{E}(\Psi_0)) \| 
\leq C_2 \| \nabla \phi_\epsilon \|_{L^1(\R^3)} \| \Psi_0 \|_{H^1},
\end{equation*}
for some positive constant $C_2$ that depends on $C_{\mathcal{E}}$.
\end{proof}

\begin{lemma}\label{lemma:reg_potentials}
Consider the regularized potential $V_{\epsilon}$. There exist three positive constants
$C_3$, $C_4$ and $C_5$ (independent of $\epsilon$) such that for any $\epsilon >0$ it holds that
\begin{equation}\label{eq:estimates_potentials}
\begin{split}
\| V_\epsilon(\Phi) \| & \leq C_3 \Bigl[ \| V \|_\infty + \| \phi_1 \|_{L^1(\R^3)} ( \| W \|_\infty + K_1 ) \Bigr] \| \Phi \|, \\
\| \nabla ( V_\epsilon(\Phi) ) \| &\leq C_4 \Bigl[ \| V \|_{1,\infty} 
+ \| \phi_1 \|_{L^1(\R^3)} ( \| W \|_{1,\infty} + K_2 ) \Bigr] \| \Phi\|_{H^1}, \\
\| \Delta ( V_\epsilon(\Phi) ) \| &\leq C_5 \Bigl[ \| V \|_{2,\infty} 
+ \| \nabla \phi_\epsilon \|_{L^1(\R^3)} ( \| W \|_{1,\infty} + K_2 ) \Bigr] \| \Phi \|_{H^2}, \\
\end{split}
\end{equation}
for any $\Phi \in Z$, where the norms $\| \cdot \|_{j,\infty}$ are the usual norms
for the spaces \linebreak $L^{\infty}(0,T;W^{j,\infty}(\Omega;\C))$ for $j=1,2$.
\end{lemma}
\begin{proof}
The three estimates are obtained following the same arguments used in Lemma \ref{lemma:reg_in_cond}
together with the new assumptions (A2b) and (A5b).
\end{proof}

With the estimates of Lemmas \ref{lemma:reg_in_cond} and \ref{lemma:reg_potentials} at hand,
we prove the following energy estimates.

\begin{theorem}[Energy estimates for the regularized problem]\label{thm:en_est_reg}
Let $G \in \Y$ and consider the solution $\Psi^a_m$ to \eqref{eq:TDKS_weak_epsilon_m}.
For almost all $t \in (0,T)$ there exist positive constants (independent of $\epsilon$) 
$\widehat{C}$, $\widehat{C}_{\nabla}$ and $K_{-1}$ and a constant $C_\epsilon>0$ (depending on $\epsilon$)
such that
\begin{equation}\label{eq:EN_1_eps}
\| \Psi^a_m(t) \|^2 \leq \exp(\widehat{C} T) \Bigl[ C_0^2 \| \phi_1 \|_{L^1(\R^3)}^2 \| \Psi_0 \|^2 + T \| G \|_{Y_{\infty,0}}^2 \Bigr],
\end{equation}
\begin{equation}\label{eq:EN_2_eps}
\| \nabla \Psi^a_m(t) \|^2 \leq \exp(\widehat{C}_\nabla T) \Bigl[ C_1^2 \| \phi_1 \|_{L^1(\R^3)}^2 \| \nabla \Psi_0 \|^2 + T \| G \|_{Y_{\infty,1}}^2 \Bigr],
\end{equation}
\begin{equation}\label{eq:EN_3_eps}
\| \Delta \Psi^a_m(t) \|^2 \leq \exp(C_\epsilon T) \Bigl[ C_2^2 \| \nabla \phi_\epsilon \|_{L^1(\R^3)}^2 \| \Psi_0 \|_{H^1}^2 + T \| G \|_{\Y}^2 \Bigr],
\end{equation}
\begin{equation}\label{eq:EN_4_eps}
\| \partial_t \Psi^a_m(t) \|_{H^{-1}} \leq K_{-1}, %(\| G \|_{Y_{\infty,1}},\| \Psi_0 \|_{H^1}),
\end{equation}
\begin{equation}\label{eq:EN_5_eps}
\| \partial_t \Psi^a_m(t) \| \leq \| \Delta \Psi^a_m(t) \| + \| V_\epsilon( \Psi^a_m(t),t) \| + \| G \|_{\Y},
\end{equation}
where $C_0$, $C_1$ and $C_2$ are given in Lemma \ref{lemma:reg_in_cond}.
Notice that $K_{-1}$ depends on $\| G \|_{Y_{\infty,1}}$ and $\| \Psi_0 \|_{H^1}$.
\end{theorem}

\begin{proof}
The estimate \eqref{eq:EN_1_eps} can be obtained as for 
\eqref{eq:EN_1}, using Lemma \ref{lemma:reg_potentials} and the first of \eqref{eq:est_in_cond}.
In fact, by testing \eqref{eq:TDKS_weak_epsilon_m} with $\Phi = \Psi_m^a$ and taking the imaginary part, we obtain
\begin{equation*}
\begin{split}
\frac{1}{2} \frac{d}{dt} \| \Psi^a_m(t) \|^2 &= \Ima( G(t) , \Psi^a_m(t))  + \Ima( V_\epsilon(\Psi_m^a) , \Psi_m^a) \\
&\leq \frac{1}{2}\| G \|_{Y_{\infty,0}}^2 + \frac{1}{2} \| V_\epsilon(\Psi_m^a) \|^2 + \| \Psi^a_m(t) \|^2.
\end{split}
\end{equation*}
Multiplying both sides by $2$ and using Lemma \ref{lemma:reg_potentials} to estimate the term $\| V_\epsilon(\Psi_m^a) \|^2$, we get
\begin{equation*}
\frac{d}{dt} \| \Psi^a_m(t) \|^2
\leq \| G \|_{Y_{\infty,0}}^2 + \| V_\epsilon(\Psi_m^a) \|^2 + 2 \| \Psi^a_m(t) \|^2
\leq \| G \|_{Y_{\infty,0}}^2 + \widehat{C} \| \Psi^a_m(t) \|^2,
\end{equation*}
where $\widehat{C}=C_3^2 \bigl[ \| V \|_\infty + \| \phi_1 \|_{L^1(\R^3)} ( \| W \|_\infty + K_1 ) \bigr]^2 + 2$.
The estimate \eqref{eq:EN_1_eps} follows by the Gr\"onwall inequality and the first 
inequality of \eqref{eq:est_in_cond}.

Let us prove \eqref{eq:EN_2_eps}. Proceeding as in Theorem \ref{thm:energy_aux}, we use
\eqref{eq:proof_EN_2} and \eqref{eq:TDKS_weak_epsilon_m} in strong form to obtain
\begin{equation*}
\frac{d}{dt} \| \nabla \Psi^a_m(t) \|^2
= 2 \Rea \Bigl[ -i ( \nabla(V_\epsilon(\Psi^a_m(t),t)) , \nabla \Psi^a_m(t) ) - i ( \nabla G(t) , \nabla \Psi^a_m(t) )\Bigr],
\end{equation*}
and the results follows using Lemmas \ref{lemma:reg_in_cond} and \ref{lemma:reg_potentials},
Poincar\'e-Friedrichs' and Gr\"onwall's inequalities.

%Next, we prove \eqref{eq:EN_4_eps}. Consider any $\Phi \in H^1(\Omega)$ such that $\| \Phi \|_{H^1}=1$
%and the equation \eqref{eq:TDKS_weak_epsilon_m}.
%We then write
%\begin{equation*}
%\begin{split}
%|(\partial_t \Psi^a_m(t),\Phi)| &\leq |(\nabla \Psi^a_m(t), \nabla \Phi)| + |(V_{\epsilon}(\Psi^a_m(t),t),\Phi)| + |(G(t),\Phi)| \\
%&\leq \|\nabla \Psi^a_m(t) \| + \| V_{\epsilon}(\Psi^a_m(t),t) \| + \| G(t) \|,
%\end{split}
%\end{equation*}
%for almost all $t \in (0,T)$.
%Using the first of \eqref{eq:estimates_potentials}, \eqref{eq:EN_1_eps} and \eqref{eq:EN_2_eps},
%we obtain that there exists a constant $K_{-1}$ independent of $\epsilon$ (but depending on $G$) such that
%\begin{equation*}
%\| \partial_t \Psi^a_m(t) \|_{H^{-1}} = \sup_{\| \Phi \|_{H^1}=1} |(\partial_t \Psi^a_m(t),\Phi)|
%\leq K_{-1}(\| G \|_{Y_{\infty,1}}).
%\end{equation*}

Let us now prove \eqref{eq:EN_3_eps}. Recalling \eqref{eq:proof_EN_3}
and using \eqref{eq:TDKS_weak_epsilon_m} in strong form, we obtain (as in \eqref{eq:DELTA})
\begin{equation*}
\frac{d}{dt} \| \Delta \Psi^a_m(t) \|^2 
= 2 \Rea \Bigl[ -i ( \Delta(V_\epsilon( \Psi^a_m(t),t)) , \Delta \Psi^a_m(t) ) -i ( \Delta G(t) , \Delta \Psi^a_m(t) ) \Bigr],
\end{equation*}
for almost all $t$ in $(0,T)$.
We can now use \eqref{eq:estimates_potentials} to write
\begin{equation*}
\begin{split}
\frac{d}{dt} \| \Delta \Psi^a_m(t) \|^2 
&\leq C_5^2 \Bigl[ \| V \|_{2,\infty} 
+ \| \nabla \phi_\epsilon \|_{L^1(\R^3)} ( \| W \|_{1,\infty} + K_2 ) \Bigr]^2 \| \Psi^a_m(t) \|_{H^2}^2 \\
&+ \| \Delta G(t) \|^2 + 2 \| \Delta \Psi^a_m(t) \|^2 \\
&\leq \| G \|_{Y_{\infty,2}}^2 + C_\epsilon \| \Delta \Psi^a_m(t) \|^2, \\
\end{split}
\end{equation*}
where $C_\epsilon = 2 + C_5^2 \Bigl[ \| V \|_{2,\infty} 
+ \| \nabla \phi_\epsilon \|_{L^1(\R^3)} ( \| W \|_{1,\infty} + K_2 ) \Bigr]^2 C_Z^2$.
The estimate \eqref{eq:EN_3_eps} follows by Gr\"onwall's inequality and Lemma \ref{lemma:reg_in_cond}.

Finally, the bounds \eqref{eq:EN_4_eps} and \eqref{eq:EN_5_eps} are obtained as for \eqref{eq:EN_4_pre} and \eqref{eq:EN_4}, respectively,
and using Lemma \ref{lemma:reg_potentials}.
\end{proof}

Notice that the energy estimates of Theorem \ref{thm:en_est_reg} have two main properties.
On the one hand, the obtained bounds are independent of $m$ (exactly as in Theorem~\ref{thm:energy_aux}).
On the other hand, only the bounds in \eqref{eq:EN_1_eps}, \eqref{eq:EN_2_eps} and \eqref{eq:EN_4_eps}
are independent of $\epsilon$, while the others depend on the norm of $\nabla \phi_\epsilon$,
which behaves as $O(1/\epsilon^\alpha)$ for some positive $\alpha$.

Since the bounds of Theorem \ref{thm:en_est_reg} are independent of $m$, we can invoke
our results of Sections \ref{sec:auxiliary} and \ref{sec:full_NL} to conclude that for any
$\epsilon >0$ there exists a unique solution $\Psi_\epsilon$ to \eqref{eq:TDKS_weak_epsilon}.
This is formally stated in the following theorem.

\begin{theorem}[Existence and uniqueness of solution to the $\epsilon$-regularized problem]\label{thm:ex_un_eps}
For any $\epsilon >0$ there exists a unique solution $\Psi_\epsilon$ to \eqref{eq:TDKS_weak_epsilon}
such that
\begin{equation*}
\Psi_\epsilon \in \Y, 
\quad 
\partial_t \Psi_\epsilon \in L^{\infty}(0,T;L^2(\Omega;\C))
\quad \text{and} \quad 
\Psi_\epsilon \in C([0,T];H^1_0(\Omega;\C)).
\end{equation*}
Moreover, there exist three positive constants $C_\alpha$, $C_\beta$ and $C_\gamma$ (independent of $\epsilon$) 
such that for any $\epsilon >0$ the following bounds hold
\begin{equation}\label{eq:no_eps_bounds}
\| \Psi_\epsilon(t) \| \leq C_\alpha,
\quad
\| \nabla \Psi_\epsilon(t) \| \leq C_\beta,
\quad
\| \partial_t \Psi_\epsilon(t) \|_{H^{-1}} \leq C_\gamma,
\end{equation}
for almost all $t$ in $(0,T)$.
\end{theorem}

\begin{proof}
The proof of existence and uniqueness of a solution $\Psi_\epsilon$ follows exactly the arguments
considered in Sections \ref{sec:auxiliary} and \ref{sec:full_NL}.
For the sake of clarity, we sketch here the main steps.
\begin{enumerate}
\item[{\rm (a)}] The energy estimates of Theorem \ref{thm:en_est_reg} allow us to show that there exists a
unique solution to the auxiliary problem obtained by replacing $F(\Psi)$ with $G \in \Y$
in \eqref{eq:TDKS_weak_epsilon}, that is,
\begin{equation}\label{eq:TDKS_weak_epsilon_AUX}
i (\partial_t \Psi,\Phi) = (\nabla \Psi, \nabla \Phi) + (V_{\epsilon}(\Psi),\Phi) + (G,\Phi) \text{ a.e. in $(0,T)$.}
\end{equation}

\item[{\rm (b)}] Given a final time $T^{\star \star}$, one defines an $H^2$-ball
$S_{B_{\circ}}(T^{\star\star})$ (analogously to \eqref{eq:27} and using \eqref{eq:EN_3_eps}),
which is centered in zero and has finite radius depending on $T^{\star\star}$,
and a map $\Phi \mapsto \mathcal{A}(\Phi):= \widetilde{\Psi}$, where $\widetilde{\Psi}$
solves the auxiliary problem \eqref{eq:TDKS_weak_epsilon_AUX} with $G$ replaced by $F(\Phi)$.
By reducing the final time $T^{\star \star}$ to a sufficiently small value, one proves
that $S_{B_{\circ}}(T^{\star\star})$ is invariant under $\mathcal{A}$ and that
$\mathcal{A} : S_{B_{\circ}}(T^{\star\star}) \rightarrow S_{B_{\circ}}(T^{\star\star})$
is a contraction (as in Lemma \ref{lemma:contraction} and Theorem \ref{thm:existence_uniqueness}).
The Banach fixed-point theorem implies
that there exists a unique solution in $S_{B_{\circ}}(T^{\star\star})$ to the
$\epsilon$-approximated problem \eqref{eq:TDKS_weak_epsilon}.

\item[{\rm (c)}] The obtained solution to \eqref{eq:TDKS_weak_epsilon} exactly satisfies the regularity
claimed in the statement of Theorem \ref{thm:ex_un_eps}. This is obtained (as in Theorem \ref{thm:existence_uniqueness}) by noticing that
the nonlinear function $F(\Psi)$ is uniformly bounded in $H^2$ for any $\Psi$ in $S_{B_{\circ}}(T^{\star\star})$.

\item[{\rm (d)}] The solution to \eqref{eq:TDKS_weak_epsilon} is extended on the time interval $[0,T]$
by repeating the previous three steps on a finite number of subintervals (of sufficiently small length
of order $T^{\star \star}$) that cover $[0,T]$.
%Notice that on each of these subintervals $F(\Psi)$ is uniformly bounded in $H^2$ for any $\Psi$ 
%belonging to a ball centered in zero and having radius depending on length of the corresponding subinterval.
\end{enumerate}

It is clear that the solution $\Psi_\epsilon$ satisfies on each subinterval considered in step {\rm (d)}
the energy estimates given in Theorem \ref{thm:en_est_reg}, upon replacement of $\Psi^a_m$ by $\Psi_\epsilon$
and $G$ by $F(\Psi_\epsilon)$, and upon redefinition of $\Psi_0$. However, these estimates  have bounds
that potentially depend on $\epsilon$, because they are obtained by estimating the $H^2$-norm of $F(\Psi_\epsilon)$.

To get $\epsilon$-independent bounds \eqref{eq:no_eps_bounds} we proceed as follows.
We define an $H^1$-ball $\widehat{S}_{B_{\circ}}(T^{\star\star})$ (analogously to \eqref{eq:27} and using \eqref{eq:EN_2_eps})
centered in zero and having finite radius depending
on $T^{\star \star}$. If necessary, we further reduce $T^{\star \star}$ to obtain 
a set $\widehat{S}_{B_{\circ}}(T^{\star\star})$ such that the intersection
$\widehat{S}_{B_{\circ}}(T^{\star\star}) \cap S_{B_{\circ}}(T^{\star\star})$ is invariant under $\mathcal{A}$.
Lemma \ref{lemma:Hartree} and \eqref{eq:Lipschitz3} guarantee that $\| F(\Lambda(t)) \|_{H^1}$ is uniformly
bounded on $\widehat{S}_{B_{\circ}}(T^{\star\star})$ by a constant independent of $\epsilon$.
Therefore, the terms $\| F(\Psi_\epsilon) \|_{Y_{\infty,0}}$ and $\| F(\Psi_\epsilon) \|_{Y_{\infty,1}}$
are bounded by some constants independent of $\epsilon$. 
Hence, the estimates \eqref{eq:no_eps_bounds} follow from Theorem \ref{thm:en_est_reg}.
In particular, Theorem \ref{thm:en_est_reg} allows us to get energy estimates (independent of $\epsilon$)
on each of subinterval of the finite covering constructed in {\rm (d)}. 
Since this covering of $[0,T]$ is finite the estimates \eqref{eq:no_eps_bounds} are obtained
by taking the maximum of the estimates over the finite number of covering subintervals. 
\end{proof}

Using Theorem \ref{thm:ex_un_eps} we prove existence of a unique solution to \eqref{eq:TDKS_weak_less_reg}.
\begin{theorem}[Existence and uniqueness of solution to the TDKS problem]\label{thm:thm_NEWNEW}
Under the assumptions {\rm (A1)}, {\rm (A2)}, {\rm (A3)}, {\rm (A2b)}, {\rm (A3b)}, {\rm (A4b)} and {\rm (A5b)}, 
there exists a unique weak solution
$\Psi \in W(0,T)$ to \eqref{eq:TDKS_weak_less_reg}.
\end{theorem}

\begin{proof}
Consider a sequence $\{\epsilon_k\}_k$ such that $\epsilon_k \rightarrow 0$.
Theorem \ref{thm:ex_un_eps} guarantees that for any $\epsilon_k$ there exists a unique solution $\Psi_{\epsilon_k}$
to the regularized problem \eqref{eq:TDKS_weak_epsilon_m}. This solution satisfies the bounds \eqref{eq:no_eps_bounds},
which are independent of $\epsilon$. Therefore, the sequence $\{ \Psi_{\epsilon_k} \}_k$ is bounded in $W(0,T)$.
This implies that there exists a subsequence $\{ \Psi_{\epsilon_{k_j}} \}_j$ that converges weakly in $W(0,T)$
and strongly in $Y$ to a limit $\widehat{\Psi}$. The continuity of functions and operators on the right-hand side of 
\eqref{eq:TDKS_weak_less_reg} allows us to show that $\widehat{\Psi}$ is a weak solution to \eqref{eq:TDKS_weak_less_reg}.
Uniqueness can be obtained by proceeding as in the proof of Theorem \ref{thm:ex_un_sol} and using \eqref{eq:Lipschitz2_new},
\eqref{eq:Lipschitz1} and \eqref{eq:Lipschitz2}.
\end{proof}

\section*{Acknowledgments}
The first author wishes to thank Prof. Joseph W. Jerome for many fruitful discussions, for his encouragement,
and for having read a first draft of this manuscript.
G. Ciaramella thanks also to Dr. Maria Infusino and Stefan Hain for having read a first draft of this manuscript
and provided several useful comments.

\bibliographystyle{plain}
\bibliography{references}

\end{document}